\documentclass[12pt]{article}

\usepackage{amssymb}
\usepackage{amsthm}
\usepackage{tikz}
\usepackage{makeidx}         
\usepackage{graphicx}        
\usepackage{multicol}        
\usepackage[bottom]{footmisc}

\usepackage{newtxtext}       %
\usepackage{newtxmath}       
\usepackage{multirow}
\usepackage{diagbox}

\usepackage{xcolor}

\usepackage[colorlinks]{hyperref}
\usepackage[noabbrev,capitalise]{cleveref}
\usepackage{booktabs}
\usepackage{geometry}

\newtheorem{assumption}{Assumption}
\newtheorem{iterative}{Iterative method}

\newtheorem{lemma}{Lemma}[section]
\newtheorem{definition}{Definition}[section]

\theoremstyle{definition}

\newtheorem{remark}{Remark}

\crefname{theorem}{theorem}{theorems}
\Crefname{theorem}{Theorem}{Theorems}
\crefname{definition}{definition}{definitions}
\Crefname{definition}{Definition}{Definitions}
\crefname{lemma}{lemma}{lemmas}
\Crefname{lemma}{Lemma}{Lemmas}
\definecolor{darkgreen}{rgb}{0.0,0.4,0.0}
\definecolor{darkred}{rgb}{0.6,0.0,0.0}
\definecolor{darkblue}{rgb}{0.0,0.0,0.5}
\definecolor{gray}{rgb}{0.5,0.5,0.5}
\definecolor{cyan}{rgb}{0.0,1.0,1.0}
\definecolor{darkcyan}{rgb}{0.0,0.5,0.5}
\definecolor{darkorange}{rgb}{0.8,0.4,0.0}
\definecolor{darkmargenta}{rgb}{0.5,0.0,0.5}
\definecolor{black}{rgb}{0.0,0.0,0.0}

\hypersetup{citecolor=blue, linkcolor=red, anchorcolor=darkblue,
	filecolor=darkblue, urlcolor=darkblue}
\hypersetup{pdfauthor={Jakob Seierstad Stokke},pdftitle={Stokke}}

\definecolor{colorL1}{RGB}{255 0 255}
\usepackage{algorithm}
\usepackage[title,titletoc]{appendix}
\usepackage[noend]{algpseudocode}
\def\BState{\State\hskip-\ALG@thistlm}

\usepackage{float} 
\usepackage{bm}
\usepackage{subcaption}

\makeindex             

\numberwithin{equation}{section}

\def \e  {\varepsilon}

\def \k  {\kappa}

\def \l  {\lambda}

\def \Om {\Omega}

\def \P {\mathbb{P}}
\def \calT {\mathcal{T}}

\def \t  {\tau}

\def \del {\nabla}
\def \div {\nabla\cdot}

\def \p  {\partial}
\def \R  {\mathbb{R}}

\def \bmu {\bm{u}}
\def \bmv {\bm{v}}

\def \bmF {\bm{f}}

\usepackage{tikz}
\usepackage{pgfplots}
\usepackage{cancel}

\newcommand{\vertiii}[1]{{\left\vert\kern-0.17ex\left\vert\kern-0.17ex\left\vert #1 
		\right\vert\kern-0.17ex\right\vert\kern-0.17ex\right\vert}}

\newcommand{\norm}[1]{{\left\vert\kern-0.25ex\left\vert\kern-0.25ex\left\vert #1 
		\right\vert\kern-0.25ex\right\vert\kern-0.25ex\right\vert}}

\makeatletter
\newsavebox{\@brx}
\newcommand{\llangle}[1][]{\savebox{\@brx}{\(\m@th{#1\langle}\)}%
	\mathopen{\copy\@brx\kern-0.5\wd\@brx\usebox{\@brx}}}
\newcommand{\rrangle}[1][]{\savebox{\@brx}{\(\m@th{#1\rangle}\)}%
	\mathclose{\copy\@brx\kern-0.5\wd\@brx\usebox{\@brx}}}
\makeatother

\usepackage{mdframed}
\usepackage{array}

\newcolumntype{C}[1]{>{\centering\arraybackslash}p{#1}}

\usepackage{nomencl}
\makenomenclature

\usepackage{etoolbox}
\renewcommand\nomgroup[1]{%
	\item[\bfseries
	\ifstrequal{#1}{A}{Physical quantities}{%
		\ifstrequal{#1}{N}{Space and time}{%
			\ifstrequal{#1}{B}{Function spaces}{%
				\ifstrequal{#1}{C}{Discretization}{%
					\ifstrequal{#1}{K}{Abbreviations}{}}}}}%
	]}

\providecommand{\keywords}[1]
{
	\small	
	\textit{Keywords---} #1
}
\usepackage{authblk}
\title{Convergent adaptive iterative schemes for solving multi-physics problems}
\author[1]{Jakob S. Stokke, Kundan Kumar, Florin A. Radu}

\affil[1]{Center for Modeling of Coupled Subsurface Dynamics, Department of Mathematics, University of Bergen, Bergen, Norway}

\date{}

\date{}
\begin{document}
\maketitle
		\begin{abstract}
			In this paper, we derive a practical, general framework for creating adaptive iterative (linearization or splitting) algorithms to solve multi-physics problems. This means that, given an iterative method, we derive \textit{a posteriori} estimators to predict the success or failure of the method. Based on these estimators, we propose adaptive algorithms, including adaptively switching between methods, adaptive time-stepping methods, and the adaptive tuning of stabilization parameters. We apply this framework to two-phase flow in porous media, surfactant transport in porous media, and quasi-static poroelasticity. 
		\end{abstract}

		\keywords{
			Multi-physics, Two-phase flow, Surfactant transport, Poroelasticity, Nonlinear, Adaptivity, A posteriori error estimates
		}

	\section{Introduction}
	\label{sec:intro}
	
	When modeling multi-physics problems, the equations are often non-linear and fully coupled, leading to complex non-linear systems. To numerically solve these systems, we often rely on iterative methods based on linearization or splitting techniques. The design of iterative methods that are robust, accurate, and efficient is very challenging. For example, in the context of degenerate parabolic PDEs like Richards' equation, the convergence of Newton's method is not guaranteed unless the time step is sufficiently small \cite{radu2006}. It is possible to improve the convergence by, e.g. considering line search and variable switching \cite{brenner_improving_2017} or trust-region techniques \cite{wang_trust-region_2013}. Another approach is to use a robust method to generate a sufficiently accurate initial guess before switching to a higher-order method. For example, the Picard method has been used to obtain initial guesses for Newton's method in~\cite{bergamaschi_mixed_1999,lehmann_comparison_1998}, and similarly, the L-scheme \cite{pop2004} has been used to initialize Newton's method in~\cite{list_study_2016,stokke_adaptive_2023,ahmed_equilibrated_2025}. Another approach applied to the Navier-Stokes equations is a combined two-step Picard-Newton algorithm \cite{pollock_analysis_2025}, utilizing the increased robustness of the Picard method. A key aspect of the switching strategy in the aforementioned papers is that robustness can be guaranteed if the method used to generate an initial guess for the higher-order method is globally convergent. This means that combining iterative methods can lead to robust, accurate, and efficient iterative algorithms provided we know when to switch between them.
	
	In this paper, we derive a practical and general framework for creating adaptive iterative (linearization or splitting)  algorithms to solve multi-physics problems based on \textit{a posteriori} error estimates. By adaptive algorithms, we mean iterative methods that change adaptively during the iteration process. There are many reasons for considering adaptive algorithms, applying both to linearization or splitting algorithms. We want to ensure convergence of the iterative procedure, increase the convergence speed, and reduce the computational time for the method. This can be done in a multitude of ways, all relying on \textit{a posteriori} error estimators. Some examples include
	\begin{itemize}
		\item derive indicators to switch between different linearization schemes. E.g. switch from a robust, typical first-order scheme to a fast, higher-order method to achieve faster convergence.
		\item derive estimators to optimize a stabilization constant which appears in many different iterative procedures, like in linearization schemes, e.g. L-scheme or modified L-scheme \cite{mitra_modified_2019} for Richards' equation, or in splitting schemes, e.g. fixed-stress or undrained-split for Biot equation.
		\item derive indicators for predicting the failure of a linearization scheme, and then automatically reduce the time step size without using computational effort to compute multiple iterations before determining that the method diverges.   
	\end{itemize} 
	The framework presented here is based on the ideas of \cite{mitra_guaranteed_2023} and \cite{stokke_adaptive_2023}. The first step is to define a measure for the linearization or splitting error for the particular scheme or multiple schemes under consideration. The second step is to derive an \textit{a posteriori} error estimator based on the previous iterations, which bounds the error. Finally, the estimator is then used to create an adaptive algorithm. See \Cref{fig: pseudo code} for an illustration of how the quantity can be used in practice. We aim to have cheap and explicit \textit{a posteriori} estimators which should not require post-processing. This means that the estimators might not necessarily be efficient in terms of providing a very sharp bound for the error, nevertheless they can be used in practice to create adaptive iterative algorithms. 
	
	Typically, \textit{a posteriori} estimators for the linearization error are used as an adaptive stopping criterion alongside spatial estimators, see e.g. \cite{mitra_guaranteed_2023,ern_adaptive_2013, heid_adaptive_2020,vohralik_posteriori_2013}. The linearization estimators in \cite{stokke_adaptive_2023} are directly compared against the linearization error to design adaptive linearization algorithms for Richards' equation. This direct comparison was also done in the proposed adaptive modified L-scheme in~\cite{javed_robust_2025}. The stabilization parameter $M$ is chosen adaptively, meaning that the algorithm can be viewed as a switch between schemes since when $M$ is large, the method is close to the L-scheme and when it becomes small, it is close to Newton's method.  In \cite{ahmed_equilibrated_2025}, it was proposed to switch between two schemes when the linearization estimator was smaller than the other estimators, including a spatial estimator. In addition, they included an adaptive stopping criterion for the iteration process.  In our case, the linearization estimator will be directly compared with the linearization error, and no spatial error estimators will be pursued.
	
	In this paper, we consider three different, important problems: two-phase flow in porous media, surfactant transport in porous media, and flow in deformable porous media (modeled by the quasi-static Biot system). Typical applications involving flow and deformation in porous media, like e.g. soil remediation or geological storage of $CO_{2}$ are spread over decades, potentially centuries, therefore the use of large time steps becomes a requirement. Hence implicit methods are commonly used to solve these problems. In addition, since the solutions to the degenerate PDEs considered here possess low regularity, see~\cite{alt1984nonstationary}, generally a backward Euler discretization is employed in time~\cite{pop2004,list_study_2016,mitra_modified_2019}. For the two-phase flow model, we consider a global and complementary pressure formulation and two different linearization schemes, the L-scheme and Newton's method. The L-scheme converges as long as the added stabilization parameter is greater than a certain critical value \cite{radu_robust_2015,radu_robust_2018}. In the Lipschitz continuous case, this critical value can be determined, but this is not possible for the H\"older continuous case. In the latter case, one has to choose relatively large values for the stabilization parameter leading to slow convergence. To avoid this, we develop two adaptive algorithms, the first an adaptive switching algorithm from the L-scheme to Newton's method, and the other an adaptive parameter tuning algorithm for optimizing the stabilization parameter. In the case of surfactant transport, we know that Newton's method struggles for fine meshes unless the time step is small enough \cite{radu2006,illiano_iterative_2021}. Therefore, we propose an adaptive switching algorithm from the L-scheme to Newton's method and an adaptive time-stepping algorithm. For the quasi-static Biot model, we consider the fixed-stress splitting scheme. We propose an adaptive parameter tuning algorithm for selecting the stabilization parameter which appears in the fixed-stress, using the theoretical bounds from \cite{storvik_optimization_2019}.
	
	The paper is structured in the following manner. We first present the framework along with computational considerations in \Cref{sec: framework}. In \Cref{sec: twophaseflow}, we consider a two-phase flow system in porous media. For the linearization of the system, the L-scheme and Newton's method are both considered. We derive estimators for predicting the success of Newton's method if the previous iterate was computed using the L-scheme or Newton's method. A similar estimator is also derived for the L-scheme. Based on these estimators we devise two algorithms to improve robustness and convergence speed. In \Cref{sec: surfactant transport}, we consider a surfactant transport system where we derive estimators for an adaptive switching algorithm and an adaptive time-stepping algorithm. In \Cref{sec: biot} we consider the quasi-static Biot system and the fixed-stress iterative splitting scheme with the goal of adaptively changing the stabilization parameter to improve the convergence.
	
	\section{Framework}\label{sec: framework}
	
	In this section, we present the general framework. Let us consider a generic multi-physics problem described by the unknown variables $a_1(\textbf{x},t),a_2(\textbf{x},t),...,a_{M}(\textbf{x},t)$ which can be both scalar and vectorial variables. If the linearization method or iterative decoupling scheme can be written in terms of a bilinear form $\mathcal{B}_{lin}$ evaluated at the difference between two sequential iterations $\delta a^{k+1}_i=a^{k+1}_{i}-a^{k}_{i}$ being equal to the residual $\mathcal{R}$ evaluated at the previous iteration $a_{i}^{k}$, the method under consideration would be a good candidate to apply the techniques developed in this paper. In other words, we consider iterative methods of the form
	\begin{equation}\label{eq: general linearization}
		\mathcal{B}_{lin}\left((\delta a_1^{k+1},\delta a_2^{k+1},...,\delta a_M^{k+1}),(\varphi_1,\varphi_2,...,\varphi_M)\right)=-\langle \mathcal{R}(a_1^{k}, a_2^{k},..., a_M^{k}),(\varphi_1,\varphi_2,...,\varphi_M)\rangle,
	\end{equation}
	with appropriate spaces for $\varphi_1,\varphi_2,...,\varphi_M$. For schemes, like the L-scheme or modified L-scheme the bilinear form is symmetric and can be defined as an iteration-dependent inner product which induces a norm. This particular norm will be used as the measure for the error of the system. In the case of non-symmetric bilinear forms, we define the iteration-dependent norm as the symmetric part of $\mathcal{B}_{lin}$, which we denote $\mathcal{B}_{lin}^{sym}$. Consequently, we measure the error in the iteration-dependent norm
	\begin{equation}
		\mathcal{B}_{lin}^{sym}\left((\varphi_1,\varphi_2,...,\varphi_M),(\varphi_1,\varphi_2,...,\varphi_M)\right):=	\norm{\varphi_1,\varphi_2,...,\varphi_M}_{inc,a_1^{k}, a_2^{k},..., a_M^{k}}^{2}.
	\end{equation}
	By denoting $\bar{a}_i^{k+1}$ as the continuous solution of \eqref{eq: general linearization} we see that
	\begin{align}
		\begin{split}
			\norm{(a_1^{k}-\bar{a}_1^{k+1}, a_2^{k}-\bar{a}_2^{k+1},..., a_M^{k}-\bar{a}_M^{k+1})}_{inc,a_1^{k}, a_2^{k},..., a_M^{k}}\leq \underbrace{\norm{\delta a_1^{k+1},\delta a_2^{k+1},...,\delta a_M^{k+1}}_{inc,a_1^{k}, a_2^{k},..., a_M^{k}}}_{\mbox{incremental error}}\\+\underbrace{\norm{(a_1^{k+1}-\bar{a}_1^{k+1}, a_2^{k+1}-\bar{a}_2^{k+1},..., a_M^{k+1}-\bar{a}_M^{k+1})}_{inc,a_1^{k}, a_2^{k},..., a_M^{k}}}_{\mbox{discretization error}},
		\end{split}
	\end{align} 
	where the first part can be interpreted as the incremental error of the system and the second as the discretization error since it measures the difference between the numerical solution at $k+1$ and the continuous solution at $k+1$. Similar identification of error components is usually sought to balance the incremental and discretization error to reduce the total number of iterations needed until convergence is achived, see e.g.~\cite{ahmed_equilibrated_2025,mitra_guaranteed_2023,ern_adaptive_2013,heid_adaptive_2020,vohralik_posteriori_2013,fevotte_adaptive_2024,fontana_posteriori_2024}.  The identification is of particular usefulness when designing an adaptive iterative method, as we are interested in predicting when it fails. We are not interested in measuring the discretization error, and will only use the incremental error as it informs us about the performance of the iterative scheme. In \cite{stokke_adaptive_2023} they considered switching between a robust scheme (the L-scheme) and a higher-order method (Newton) based on \textit{a posteriori} error estimates for the incremental error.  Here, we will show how to derive computable estimators in a similar manner for multiple problems. The design of adaptive iterative algorithms will rely upon these bounds. To bound the incremental error by an estimator, we will follow similar steps as in \cite{stokke_adaptive_2023}, namely
	\begin{itemize}
		\item[Step 1.] Recognizing that the incremental error satisfies
		\begin{align}\label{eq: framework step 1}
			\begin{split}
				\norm{\delta a_1^{k+1},...,\delta a_M^{k+1}}_{inc,a_1^{k},..., a_M^{k}}^{2}=&-\langle \mathcal{R}(a_1^{k}, ..., a_M^{k}),(\delta a_1^{k+1},...,\delta a_M^{k+1})\rangle\\
				&-\mathcal{B}_{lin}^{non-sym}\left((\delta a_1^{k+1},...,\delta a_M^{k+1}),(\delta a_1^{k+1},...,\delta a_M^{k+1})\right),  
			\end{split}
		\end{align}
		where $\mathcal{B}_{lin}^{non-sym}$ is the non-symmetric part of the bilinear form.
		\item[Step 2.]  \eqref{eq: framework step 1} can be expanded by using that the previous iterate satisfies \eqref{eq: general linearization} at $k-1$, such that we get
		\begin{align}\label{eq: framework step 2}
			\begin{split}
				\norm{\delta a_1^{k+1},...,\delta a_M^{k+1}}_{inc,a_1^{k},..., a_M^{k}}^{2}=&-\langle \mathcal{R}(a_1^{k}, ..., a_M^{k})-\mathcal{R}(a_1^{k-1}, ..., a_M^{k-1}),(\delta a_1^{k+1},...,\delta a_M^{k+1})\rangle\\
				&+\mathcal{B}_{lin}\left((\delta a_1^{k},...,\delta a_M^{k}),(\delta a_1^{k+1},...,\delta a_M^{k+1})\right)\\&-\mathcal{B}_{lin}^{non-sym}\left((\delta a_1^{k+1},...,\delta a_M^{k+1}),(\delta a_1^{k+1},...,\delta a_M^{k+1})\right).
			\end{split}
		\end{align}
		\item[Step 3.] While ignoring the non-symmetric part, see which terms cancel in \eqref{eq: framework step 2} and through the use of inequalities along with algebraic manipulation try to recover the incremental error. 
		\item[Step 4.] The non-symmetric part of \eqref{eq: framework step 1} usually needs to be considered in more detail, and possibly additional assumptions are going to be needed, see \Cref{sec: twophaseflow} and \Cref{sec: surfactant transport}. The goal is again to recover the incremental error.
	\end{itemize}
	

	

	\begin{figure}[H]
		\centering
		\includegraphics[width=0.9\textwidth]{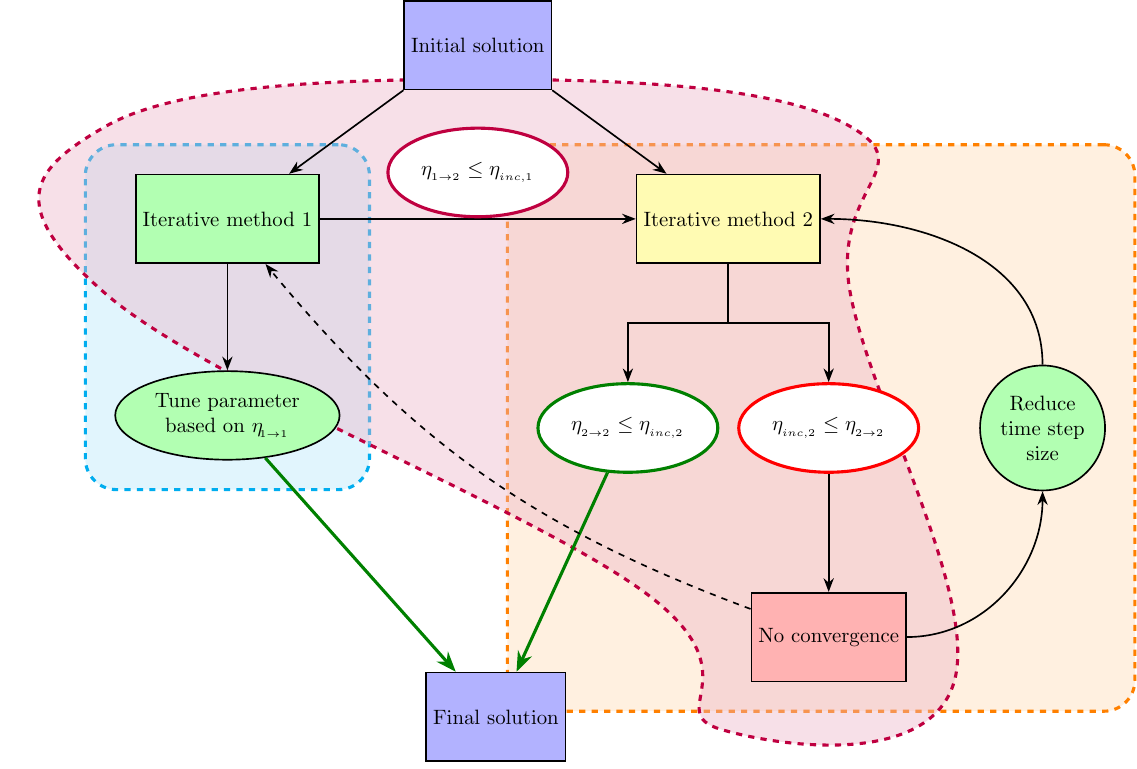}
		\caption{Illustration of different adaptive iterative algorithms based on \textit{a posteriori} error estimators. Here $\eta_{\!_{\,i\to j}}$ represents the bound on the incremental error $\eta_{\!_{\,inc,i}}$ using the iterate of scheme $i$ to predict the error at the next iteration computed using scheme $j$. Each colored outline illustrates three different algorithms; \textbf{blue}: adaptively tuning the stabilization parameter in iterative method 1, \textbf{red}: adaptively switching between iterative method 1 and 2, and \textbf{orange}: adaptive time stepping to ensure convergence of iterative method 2. }
		\label{fig: pseudo code}
	\end{figure}
	
	\subsection{Notation}
	Throughout the paper, we will use common notations from functional analysis. Let $\Omega\subset\mathbb{R}^{d},\,d\in\{1,2,3\}$ be a bounded domain with a Lipschitz continuous boundary. The space of square integrable functions is denoted by $L^{2}(\Omega)$ with $\langle\cdot,\cdot\rangle$ and $\|\cdot\|$ being the usual inner product and norm, respectively. The space with weak derivatives of order $m$ in $L^{2}(\Omega)$ will be the Sobolev space $H^{m}(\Omega)$ of $L^{2}(\Omega)$. By $H_0^{1}(\Omega)$ we denote the subspace of $H^{1}(\Omega)$ with vanishing trace at the boundary. The dual of $H_0^{1}(\Omega)$ will be denoted $H^{-1}(\Omega)$. Spaces written in bold will be vector-valued. For the discretization in space throughout the paper, we will consider a triangulation of $\Omega$ into simplices denoted $\calT_h$, where $h:=\max_{K\in\calT_h}\{diam(K)\}$ is the mesh size. We will consider the following finite element spaces
	\begin{align}
		Q_h :=& \left\{q_h\in H_0^{1}(\Omega)\,\big|\, q_{h|K}\in \P_{1}(K),\,K\in\mathcal{T}_h\right\},\\
		\bm{V}_h :=& \left\{\bmv_h\in \bm{H}_0^{1}(\Omega)\,\big|\, \bmv_{h|K}\in \P_{2}(K),\,K\in\mathcal{T}_h\right\},
	\end{align}
	where $\P_{p}(K)$ is the space of $p$-order polynomials on $K$. In time, we decompose the interval $[0, T]$ for $T>0$ into the discrete times $0<t_1<\cdots < t_{N-1}< T$ where the size of each interval, i.e. the time step size, is denoted by $\t$. 
	
	
	Related to the estimators, we introduce the following notation. We denote the difference between two consecutive iterates by $\delta a^{k+1}:=a^{k+1}-a^{k}$. For two iterative schemes $i$ and $j$, we denote by $\eta_{\!_{\,i\to j}}$ the bound on the error if a switch from $i$ to $j$ happens. Throughout this paper if $i<j$, then $i$ is a more robust scheme and $j$ is either a higher-order scheme or has a faster convergence rate. In the case when $i=j$, the estimator predicts the bound on the next iterate using the same scheme, i.e. the failure or success of a method. The incremental error for a method $i$ will be denoted by $\eta_{\!_{\,inc,i}}^{k}$. 
	
	
	\subsection{Computational considerations}\label{sec: comput considerations}
	
	We are interested in computationally inexpensive estimators so that the efficiency considerations are less important. With further computational efforts, e.g., by post-processing of the numerical simulations, the efficiency can be improved. We point out \cite{ahmed_equilibrated_2025,ern_adaptive_2013,fevotte_adaptive_2024,fontana_posteriori_2024} for ideas related to the post-processing to improve efficiency in the examples that we consider, but we do not pursue this in our algorithms discussed here. Moreover, our estimators may be valid only in a subset of the domain. For example, in some of the estimates below, we will divide by a quantity which potentially is zero, limiting its validity in a region where it is non-zero. Again, this may be avoided by additional computational efforts. For example, in \cite{stokke_adaptive_2023}, they introduced equilibrated fluxes that avoid these regions. In terms of our emphasis on the cheapness, it is relevant to mention that    only computing the estimators outside of these regions, however, did not influence the performance of the adaptive switching. Consequently, in what follows, we only compute the estimators in the region of validity. We observe in the numerical experiments that this choice still gives good convergence properties. Since this is a practical framework, we also make the choice that any constants appearing in the estimates which are related to assumptions will not be computed. This means that the constants in the assumptions in \Cref{sec: twophaseflow} and \Cref{sec: surfactant transport} are set to zero in the actual computations.
	
	When we seek to switch between schemes, we can use $\eta_{\!_{\,i\to j}}^{k}\leq \eta_{\!_{\,inc,i}}^{k}$ to switch form $i$ to $j$. Instead we introduce a constant $C_{\rm tol}\geq 1$ to speed up the switching and will make the switch when $\eta_{\!_{\,i\to j}}^{k}\leq C_{\rm tol}\eta_{\!_{\,inc,i}}^{k}$. Additionally, we use $\eta_{\!_{\,j\to j}}^{k}\geq\eta_{\!_{\,inc,j}}^{k}$ as a safety mechanism to return quickly to the robust method if the switch happens to soon.  An illustration of the adaptive switching strategy is outlined in red in~\Cref {fig: pseudo code}. In addition, if one chooses to stick to a robust method (Iterative method 1) with a stabilization parameter, the estimators can be used for tuning the parameter (algorithm outlined in blue). When solely using a higher-order method (Iterative method 2), reducing the time step size based on the estimators can be sought to improve the convergence (algorithm outlined in orange).  However, in the case of adaptive time-stepping, we consider halving the time step when $\eta_{\!_{\,i\to i}}\geq 1$, as in general this is enough to capture the failure within one iteration and avoids the worst-case scenario of divergence of the method. Furthermore, a direct comparison between $\eta_{\!_{\,i\to i}}^{k}\sim\eta_{\!_{\,inc,i}}^{k}$ might lead to overrefinement of the time step size due to the estimator not necessarily being smaller than the current incremental error despite the method converging. Whereas for adaptive switching the direct comparison can lead to a premature return to the more robust method, but then convergence is guaranteed. We also increase the time step size again when the scheme $i$ converges in less than 5 or 10 iterations. This choice for increasing the time step size is mainly heuristic. We could devise a strategy for increasing $\t$ based on the ratio $\eta_{\!_{\,i\to i}}^{k}/\eta_{\!_{\,inc,i}}^{k}$. For example, if it is small enough, we should increase the time step size again, but this is not pursued here.

	The estimators  $\ \eta_{\!_{\,i\to i}}^k$ predict the incremental error $\eta_{\rm inc,i}^{k+1}$ of the $(k+1)^{\rm th}$ iteration if done using the iterative method $i$. Similar to e.g.~\cite{ahmed_equilibrated_2025,mitra_guaranteed_2023,fevotte_adaptive_2024} to asses the sharpness of the estimate we use the \textbf{effectivity index} defined as the ratio of the estimator to the error, i.e. if $(k+1)^{\rm th}$ iteration is Iterative method $i$ then
	\begin{align}\label{eq: eff ind}
		\text{(Eff. Ind.)}_k:=
		\eta_{\!_{\,i\to i}}^k/\eta_{\rm inc,i}^{k+1}, \text{ if } k^{\rm th} \text{ iteration is Iterative method $i$}.
	\end{align}
	The indices should always be greater than 1, and the closer they are to 1, the sharper the estimate is. However, we have chosen not to compute constants in the estimates and therefore cannot guarantee that the computed estimators are an upper bound on the incremental error. Even when we have a guaranteed upper bound, the bound does not need to be sharp. This should be accounted for when designing the adaptive algorithms.

	All linear systems are solved with a direct solver, and the implementation is done in FreeFEM~\cite{freefem}.
	
	\section{Two-Phase flow in porous media}\label{sec: twophaseflow}
	
	In this section, we consider a two-phase flow model in porous media. The model we consider assumes that the fluids are immiscible and incompressible, and that the solid matrix is non-deformable. We will adopt a global pressure and complementary pressure formulation, which is obtained through the Kirchhoff transformation \cite{arbogast_existence_1992,chavent_mathematical_1986,chen2001}. Let $\alpha\in\{w,n\}$ denote the wetting and non-wetting phases, then the two-phase model can be written as; Find saturation $(s_\alpha)$, pressure $(p_\alpha)$, flux $(\boldsymbol{q}_\alpha)$ and density $(\rho_\alpha)$ of the phase $\alpha$ such that
	\begin{subequations}
		\begin{align}
			\partial_t (\phi\rho_\alpha s_\alpha)+\div(\rho_\alpha\boldsymbol{q}_\alpha)&=0, \quad&&\alpha\in\{w,n\},\\
			\boldsymbol{q}_{\alpha}&=-\frac{k_{\!_{r,\alpha}}\k}{\mu_\alpha}(\del p_\alpha-\rho_\alpha\boldsymbol{g}),\quad&&\alpha\in\{w,n\},\\
			s_w+s_n&=1,\\
			p_n-p_w &= p_{c}(s_w),
		\end{align}
	\end{subequations}
	where $\phi$ is the porosity, $\kappa$ is the permeability, $\mu_\alpha$ is the viscosity of phase $\alpha$ and $\bm{g}$ is the gravitational vector. Also,  $k_{\!_{r,\alpha}}(\cdot)$ is the relative permeability and $p_{c}(\cdot)$ is the capillary pressure, both are given functions of the saturation. 
	
	We adopt a global and complementary pressure formulation by introducing the transformations
	\begin{align}
		P(\boldsymbol{x},s_w):=&\, p_n(\boldsymbol{x})-\int_{0}^{s_w}f_{w}(\boldsymbol{x},\xi)\frac{\partial p_c}{\partial \xi}(\boldsymbol{x},\xi) d\xi,\\
		\Theta(\boldsymbol{x},s_w):=&\,-\int_{0}^{s_w}f_{w}(\boldsymbol{x},\xi)\lambda_n(\boldsymbol{x},\xi)\frac{\partial p_c}{\partial \xi}(\boldsymbol{x},\xi) d\xi,
	\end{align}
	where the mobility of phase $\alpha$ is denoted by $\lambda_{\alpha}=k_{\!_{r,\alpha}}/\mu_{\alpha}$. Also, $f_w$ is the fractional flow function, i.e. $f_w=\lambda_w/(\lambda_w+\lambda_n)$.
	For more details about the transformation, including the existence and uniqueness of a weak solution to the mixed formulation we refer to \cite{chen2001}. We will consider a conformal Galerkin discretization, but the result follows in a similar manner when using a mixed method. The transformation leads to the two-field formulation
	\begin{subequations}\label{eq: two phase complementary}
		\begin{align}
			\partial_t s(\Theta)-\div(\del\Theta)-\div(f_w(s)\k{\lambda_{t}}(s)\del P)&=f_1(s),\quad&&\mbox{in }\Omega\times[0,T],\label{eq: two phase complementary 1}\\
			-\div(\k{\lambda_{t}}(s)\del P)&=f_2(s),\quad&&\mbox{in }\Omega\times[0,T],
		\end{align}
		with $s:=s_w$ and $\lambda(s)=\lambda_n+\lambda_w$ being the total mobility. 
	\end{subequations}
	\begin{assumption}\label{as:auxiliary_functions}
		For the material properties $s(\cdot)$, $\k$, $f_w(\cdot)$ and $\l_t(\cdot)$, and source terms $f_1(\cdot),f_2(\cdot)$ in \eqref{eq: two phase complementary}, the following assumptions are made:
		\begin{itemize}
			\item[(a)] The saturation function $s:\R\to \R$, $s(0)=0$ is strictly increasing and H\"older contintous with exponent $\gamma\in(0,1]$.
			\item[(b)] The permeability $\k$ is strictly positive, and the total mobility $\l_{t}(\cdot):\R\to\R$ satifies the growth condition
			\begin{equation*}
				|\l_t(s(\Theta_1))-\l_t(s(\Theta_2))|^{2}\leq C\langle s(\Theta_1)-s(\Theta_2),\Theta_1-\Theta_2\rangle,\quad\forall\,\Theta_1,\Theta_2\in\R,
			\end{equation*}
			and there exits constants $\l^{M}_{t},\l^{m}_{t}>0$ such that for all $y\in\R$, one has $0<\l_{t}^{m}\leq \l_{t}(y)\leq \l_{t}^{M}<\infty$.
			\item[(c)] The fractional flow function $f_w:\R\to\R$ is bounded, $f_{w}(0)=0$ and $f_w(1)=1$, and is satifies the growth condition
			\begin{equation*}
				|f_w(s(\Theta_1))-f_w(s(\Theta_2))|^{2}\leq C\langle s(\Theta_1)-s(\Theta_2),\Theta_1-\Theta_2\rangle,\quad\forall\,\Theta_1,\Theta_2\in\R.
			\end{equation*}
			If $f_1,f_2$ depend on $s$, they satisfy a similar growth condition. 
			\item[(d)] There exists a constant $M_P<\infty$ such that $\|\del P_h^{n}\|_{L^{\infty}}\leq M_p$ for all $n\in\mathbb{N}$.
			\item[(e)] The source functions satisfy $f_1,f_2\in C(0,T;L^2(\Om)).$
		\end{itemize}
	\end{assumption}
	\begin{remark}
		The first assumption \ref{as:auxiliary_functions}~(a) covers many applications of practical interest. For example, the Brooks-Corey type models has a capillary pressure function which behaves like $p_{c}(s)\propto s^{-1/\beta}$ for $\beta>1$. If $s\to 0$, then the water mobility scales asymptotically like $\lambda_{w}(s)\propto s^{\alpha}$. This in turn implies that the saturation function $s(\Theta)\propto\Theta^{\gamma}$ where $\gamma=\frac{\beta}{\beta\alpha-1}$ becomes H\"older continuous if $\gamma<1$.  The assumption that $s(\cdot)$ is strictly increasing is used to show existence and uniqueness of a solution~\cite{chen2001}. In addtion, it is needed for Assumptions \ref{ass: 1} and \ref{ass: 2} to hold, which are required for the \textit{a posteriori} error estimates in this section. The growth conditions in assumption \ref{as:auxiliary_functions}~(b)-(c) are satisfied in most realistic scenarios. \ref{as:auxiliary_functions}~(d) is valid if the data is sufficiently regular, and is needed to show the convergence of the L-scheme. 
	\end{remark}
	
	The fully discrete problem of \eqref{eq: two phase complementary} at time $t_n$ using the Backward Euler method in time and $\P 1$ elements in space is: Given $\Theta_{h}^{n-1},P_{h}^{n-1}\in Q_h$ find $\Theta_{h}^{n},P_{h}^{n}\in Q_h$ such that
	\begin{subequations}\label{eq: two phase discrete}
		\begin{align}
			\langle s(\Theta_{h}^{n})-s(\Theta_{h}^{n-1}),q_h\rangle +\tau\langle\del\Theta_{h}^{n},\del q_h\rangle+\tau\langle f_w(s(\Theta_{h}^{n}))\k{\lambda_{t}}(s(\Theta_{h}^{n}))\del P^{n},\del q_h\rangle&=\tau\langle f_1^{n},q_h\rangle,\label{eq: two phase discrete a}\\
			\langle \k{\lambda_{t}}(s(\Theta_{h}^{n}))\del P^{n},\del r_h\rangle&=\langle f_2^{n},r_h\rangle, \label{eq: two phase discrete b}
		\end{align}
		for all $q_h, r_h\in Q_h$.
	\end{subequations} For simplicity $f_1,f_2$ does not depend on $s(\cdot)$ in the following. To linearize the above problem we will consider two schemes, the L-scheme \cite{radu_robust_2015,radu_robust_2018} and Newton's method. Due to the degeneracy of the equations, Newton's method is not guaranteed to converge. The L-scheme offers a robust alternative, but the convergence is only linear. Another issue is that $s(\cdot)$ may only be H\"older continuous, and for smaller H\"older exponents, it is known that the stabilization parameter in the L-scheme has to be picked sufficiently large to guarantee convergence \cite{radu_robust_2018}. A larger parameter will likely lead to even slower convergence.  Therefore, our goal in this section is to derive \textit{a posteriori} error estimators for predicting when to tune the stabilization parameter $L$ and also for when to switch between the L-scheme and Newton's method adaptively. The strategies presented here can be combined, but we chose not to for simplicity. Note that the adaptive regularization scheme in \cite{fevotte_adaptive_2024} is also able to deal with H\"older continuous nonlinearities. 
	Then we can define the residual of the above formulation as
	\begin{definition}[Residual two-phase flow] For all $q_h,r_h\in Q_h$, the total residual of \eqref{eq: two phase discrete a} and \eqref{eq: two phase discrete b} is defined by
		\begin{equation}
			\begin{aligned}
				\langle\mathcal{R}(\varphi_1,\varphi_2),(q_h,r_h)\rangle :=& \langle s(\varphi_1)-s(\Theta_{h}^{n-1}),q_h\rangle+\tau\langle\del\varphi_1,\del q_h\rangle+\tau\langle f_w(s(\varphi_1))\k{\lambda_{t}}(s(\varphi_1))\del\varphi_2,q_h\rangle\\
				&+\tau\langle \k{\lambda_{t}}(s(\varphi_1))\del\varphi_2,r_h\rangle-\tau\langle f_1^{n},q_h\rangle-\tau\langle f_2^{n},r_h\rangle.
			\end{aligned}
		\end{equation}
		\label{def: twophase residual}
	\end{definition}
	
	\subsection{Linearization methods}
	We consider two linearization schemes, the L-scheme and Newton's method. Let $k\geq 1$ be the iteration index, and we omit the time index at the new time, i.e. we let $\Theta_{h}^{k}:=\Theta^{n,k}_h$ and $P_{h}^{k}:=P^{n,k}_{h}$. 
	First we consider the L-scheme:  Given $\Theta_{h}^{n-1},\Theta_{h}^{k},P_{h}^{k}\in Q_h$ find $\Theta_{h}^{k+1},P_{h}^{k+1}\in Q_h$ for an $L>0$ such that
	\begin{subequations}\label{eq: two phase Lscheme}
		\begin{align}
			\begin{split}
				\langle L(\Theta_{h}^{k+1}-\Theta_{h}^{k}),q_h\rangle+\langle s(\Theta_{h}^{k})-s(\Theta_{h}^{n-1}),q_h\rangle +\tau\langle\del\Theta_{h}^{k+1},\del q_h\rangle\\
				+\tau\langle f_w(s(\Theta_{h}^{k}))\k{\lambda_{t}}(s(\Theta_{h}^{k}))\del P_{h}^{k+1},\del q_h\rangle&=\tau\langle f_1^{n},q_h\rangle,
			\end{split}\\
			\tau\langle \k{\lambda_{t}}(s(\Theta_{h}^{k}))\del P_{h}^{k+1},\del r_h\rangle&=\tau\langle f_2^{n},r_h\rangle,
		\end{align}
	\end{subequations}
	for all $ q_h, r_h\,\in Q_h$. By \cite[Theorem 4.1]{radu_robust_2018} if \Cref{as:auxiliary_functions} (b)-(d) hold, the convergence of the L-scheme is guaranteed when $s$ is Lipschitz continuous case if $L\geq L_s$ with $L_s$ being the Lipschitz constant. If $s$ is only H\"older continuous, the L-scheme converges numerically when $L$ is chosen large enough to decrease the contribution of the additional right-hand side term in \cite[Eq. (4.25)]{radu_robust_2018}, i.e. $C_{\!_{L,\tau}}\times L^{\frac{2}{\gamma-1}}$ with $C_{\!_{L,\tau}}$ depending on both $L$ and $\tau$. 
	Then we can define the bilinear form for the L-scheme depending on the previous iteration $\Theta_{h}^{k}$,
	\begin{equation}\label{eq: two phase bilinear form Lscheme}
		\begin{aligned}
			\mathcal{B}_{L,\Theta_{h}^{k}}((\varphi_1,\varphi_2),(q_h,r_h)):=\langle L\varphi_1,q_h\rangle+\tau\langle \del \varphi_1,\del q_h\rangle+\tau\langle f_w(s(\Theta_{h}^{k}))\k{\lambda_{t}}(s(\Theta_{h}^{k}))\del\varphi_2,\del q_h\rangle \\
			+ \t\langle \k{\lambda_{t}}(s(\Theta_{h}^{k}))\del \varphi_2,\del r_h\rangle.
		\end{aligned}
	\end{equation}
	Consequently, we can define the L-scheme linearization as 
	\begin{iterative}[L-scheme] For the bilinear form defined in \eqref{eq: two phase bilinear form Lscheme} and the residual in \Cref{def: twophase residual} the L-scheme can be defined as
		\begin{equation}\label{eq: twophase Lscheme}
			\mathcal{B}_{L,\Theta_{h}^{k}}((\delta\Theta_{h}^{k+1},\delta P_{h}^{k+1}),(q_h,r_h))=-\langle\mathcal{R}(\Theta_{h}^{k},P_{h}^{k}),(q_h,r_h)\rangle.
		\end{equation}
		
	\end{iterative}
	Then we also define an iteration-dependent norm for the L-scheme as the symmetric part of the bilinear form \eqref{eq: two phase bilinear form Lscheme}
	\begin{definition}[Iteration dependent norm for L-scheme]  For $\varphi_1,\varphi_2\in H_{0}^{1}(\Omega)$, the iteration-dependent norm for the L-scheme \eqref{eq: twophase Lscheme} is defined by
		\begin{equation}
			\norm{\varphi_1,\varphi_2}_{L,\Theta^{n,i}}=\left(\int_{\Omega}L\varphi_1^{2}+\tau|\del\varphi_1|^{2}+\t|(\k{\lambda_{t}}(s(\Theta^{n,i})))^{\frac{1}{2}}\del\varphi_2|^{2}\right)^{\frac{1}{2}}.
		\end{equation}
	\end{definition}	
	\begin{remark}
		The L-scheme defined in \eqref{eq: two phase Lscheme} can be viewed as a pseudo-time-stepping scheme, since $L(\Theta_{h}^{n,k+1}-\Theta_{h}^{n,k})$ is in essence a time derivative with time step size $\tau=1/L$. When the L-scheme converges, the pseudo-time derivative reaches a steady state.  Such pseudo-time-stepping methods have been used in the CFD community~\cite{klaij2006pseudo}. Choosing different time steps at different locations leads to a spatially varying $L$. This means $L$ can be replaced by a diagonal matrix $\bm{M}$, which is closely related to the modified L-scheme \cite{mitra_modified_2019} where the entries in $\bm{M}$ can be chosen adaptively~\cite {javed_robust_2025}.
	\end{remark}
	Newton's method can be written as: Given $\Theta_{h}^{n-1},\Theta_{h}^{k},P_{h}^{k}\in Q_h$ find $\Theta_{h}^{k+1},P_{h}^{k+1}\in Q_h$ such that
	\begin{subequations}\label{eq: newtons method}
		\begin{align}
			\begin{split}
				\langle s'(\Theta_{h}^{k})(\Theta_{h}^{k+1}-\Theta_{h}^{k}),q_h\rangle+\langle s(\Theta_{h}^{k})-s(\Theta_{h}^{n-1}),q_h\rangle +\tau\langle\del\Theta_{h}^{k+1},\del q_h\rangle&\\
				+\tau\langle (f_w\circ s)'(\Theta_{h}^{k})\k{\lambda_{t}}(s(\Theta_{h}^{k}))\del (P_{h}^{k})\delta \Theta_{h}^{k+1},\del q_h\rangle
				\\+\tau\langle f_w(s(\Theta_{h}^{k}))\k(\lambda\circ s)'(\Theta_{h}^{k}))\del (P_{h}^{k})\delta \Theta_{h}^{k+1},\del q_h\rangle\\+\tau\langle f_w(s(\Theta_{h}^{k}))\k{\lambda_{t}}(s(\Theta_{h}^{k}))\del P_{h}^{k+1},\del q_h\rangle&=\tau\langle f_1^{n},q_h\rangle,
			\end{split}\\
			\tau\langle \k(\lambda\circ s)'(\Theta_{h}^{k})\del (P_{h}^{k}) \delta \Theta_{h}^{k+1},\del r_h\rangle+\tau\langle \k{\lambda_{t}}(s(\Theta_{h}^{k}))\del P_{h}^{k+1},\del r_h\rangle&=\t\langle f_2^{n},r_h\rangle,
		\end{align}
	\end{subequations}
	for all $ q_h, r_h\,\in Q_h$. The convergence of Newton's method for degenerate parabolic problems is generally not guaranteed, as the Jacobian may become singular. Even for the regularized Kirchhoff-transformed Richards' equation considered in~\cite{radu2006}, convergence of Newton's method can only be guaranteed under strict conditions on the discretization parameters, i.e., $\t$ is on the order of $h^{d}$, where $d$ is the dimension.
	Then we can similarly define the bilinear form for Newton's method 
	\begin{equation}\label{eq: two phase flow bilinear form Newton}
		\begin{aligned}
			\mathcal{B}_{N,\Theta_{h}^{k}}((\varphi_1,\varphi_2),(q_h,r_h)):=&\langle s'(\Theta_{h}^{k})\varphi_1,q_h\rangle+\tau\langle (f_w\circ s)'(\Theta_{h}^{k})\k{\lambda_{t}}(s(\Theta_{h}^{k}))\del(P_{h}^{k}) \delta \varphi_1,\del q_h\rangle\\
			&+\tau\langle f_w(s(\Theta_{h}^{k}))\k(\lambda\circ s)'(\Theta_{h}^{k}))\del (P_{h}^{k})\delta \varphi_1,\del q_h\rangle+\tau\langle \del \varphi_1,\del q_h\rangle\\
			&+\tau\langle f_w(s(\Theta_{h}^{k}))\k{\lambda_{t}}(s(\Theta_{h}^{k}))\del\varphi_2,\del q_h\rangle \\
			&+\tau\langle \k(\lambda\circ s)'(\Theta_{h}^{k})\del  (P_{h}^{k})\varphi_1,\del r_h\rangle+ \tau\langle \k{\lambda_{t}}(s(\Theta_{h}^{k}))\del \varphi_2,\del r_h\rangle.
		\end{aligned}
	\end{equation}
	We can now write the Newton method as 
	\begin{iterative}[Newton's method] For the bilinear form defined in \eqref{eq: two phase flow bilinear form Newton} and the residual in \Cref{def: twophase residual} the Newton's method can be defined as
		\begin{equation}\label{eq: twophase Newton}
			\mathcal{B}_{N,\Theta_{h}^{k}}((\delta\Theta_{h}^{k+1},\delta P_{h}^{k+1}),(q_h,r_h))=-\langle\mathcal{R}(\Theta_{h}^{k},P_{h}^{k}),(q_h,r_h)\rangle.
		\end{equation}
	\end{iterative}
	\noindent We further define the iteration-dependent norm for Newton's method as  the symmetric part of the bilinear form \eqref{eq: two phase flow bilinear form Newton}
	\begin{definition}[Iteration dependent norm for Newton's method]  For $\varphi_1,\varphi_2\in H_{0}^{1}(\Omega)$, the iteration-dependent norm for Newton's method \eqref{eq: twophase Newton} is defined by
		\begin{equation}\label{eq: twophase newton norm}
			\norm{\varphi_1,\varphi_2}_{N,(\Theta_{h}^{k},P_{h}^{k})}=\left(\int_{\Omega}s'(\Theta_{h}^{k})\varphi_1^{2}+\tau|\del\varphi_1|^{2}+\tau|(\k{\lambda_{t}}(s(\Theta_{h}^{k})))^{\frac{1}{2}}\del\varphi_2|^{2}\right)^{\frac{1}{2}}.
		\end{equation}
	\end{definition}
	Note that \Cref{as:auxiliary_functions} (a)-(b) ensures that \eqref{eq: twophase newton norm} is a norm if $s(\cdot)$ is Lipchitz continuous, even in the case when $s'(\cdot)=0$ in parts of the domain. However, since we only assume H\"older continuity, $s'(\cdot)$ may blow up and needs to be assumed to be finite for the quantity above to be a norm. To bound the incremental error in this norm below requires assumptions which implicitly infer that $s'(\cdot)$ does not blow up.
	
	\subsection{Estimators}\label{sec: twophase estimators}
	Based on the previous section we have an error measure for the incremental error of both Newton's method and the L-scheme in an iteration-dependent norm. Now we can derive \textit{a posteriori} estimators in a similar fashion to the steps outlined previously. We will need the following assumptions.
	
	\begin{assumption}\label{ass: 1}
		For a $k\in\mathbb{N}$, there exists a constant $C_{1}^{k}\in [0,2)$ such that
		\begin{equation*}
			\tau  \left|\left((f_w\circ s)'(\Theta_{h}^{k})\k{\lambda_{t}}(s(\Theta_{h}^{k}))+f_w(s(\Theta_{h}^{k}))\k(\lambda\circ s)'(\Theta_{h}^{k})\right)\del P_{h}^{k}\right|^{2}\leq \left(C_{1}^{k}\right)^{2}s'(\Theta_{h}^{k})/4,
		\end{equation*}
		almost everywhere in $\Omega$.
	\end{assumption}
	\begin{assumption}\label{ass: 2}
		For a $k\in\mathbb{N}$, there exists a constant $C_{2}^{k}\in [0,2)$ such that
		\begin{equation*}
			\tau\left|(\k{\lambda_{t}}(s(\Theta_{h}^{k})))^{-\frac{1}{2}}(\lambda_{t}\circ s)'(\Theta_{h}^{k})\del P_{h}^{k}\right|^{2}\leq \left(C_{2}^{k}\right)^{2}s'(\Theta_{h}^{k})/4,
		\end{equation*}
		almost everywhere in $\Omega$.
	\end{assumption}
	Observe that the constants in \Cref{ass: 1} and \Cref{ass: 2} are fully computable as both $\Theta_h^k$ and $P_h^k$ are known. The assumptions hold if the numerical fluxes are bounded, the time step size $\t$ is small and if the saturation is strictly increasing. However, \Cref{as:auxiliary_functions} (a) allows for $s(\cdot)$ to only be H\"older continuous; then $s'(\cdot)$ may blow up in parts of the domain, in which case the assumptions no longer hold. In addition, the inequalities are always satisfied in the degenerate case when $s'(\Theta_h^k)=0$ as $(s'(\Theta_h^k))^{2}$ appears on the left-hand side of the inequality.  For all of the estimators we also assume that the fractional flow function satisfies $|f_w|\leq C_{\!_{f_w}}<1$, meaning there is always some residual non-wetting fluid.
	\begin{lemma}[L-scheme to Newton estimator]\label{lem: twophase L to N}
		Let Assumptions \ref{as:auxiliary_functions},\ref{ass: 1} and \ref{ass: 2} hold, and the fractional flow function be bounded by $C_{\!_{f_w}}<1$. Let $\{\Theta_{h}^{k},P_{h}^{k}\}$ be a sequence of iterates generated using the L-scheme \eqref{eq: twophase Lscheme}. Then, if $\,\hat{\Theta}^{k+1},\hat{P}^{k+1}$ are computed using Newton's method \eqref{eq: twophase Newton}, the incremental error satisfies
		\begin{equation}
			\norm{(\hat{\Theta}^{k+1}-\Theta_{h}^{k},\hat{P}^{k+1}-P_{h}^{k})}_{N,(\Theta_{h}^{k},P_{h}^{k})}\leq \eta_{\!_{\,1\to 2}}^{k},
		\end{equation}
		where
		\begin{align}\label{eq: twophase est 1to2}
			\eta_{\!_{\,1\to 2}}^{k}=\frac{2}{2-\max(C_1^{k},C_2^{k},2C_{\!_{f_w}})}\left([\eta_s^{k}]^{2}+\tau[\eta_\Theta^{k}]^{2}+\t[\eta_{\lambda_{t}}^{k}]^{2}\right)^{\frac{1}{2}},
		\end{align}
		with
		\begin{equation*}
			\begin{aligned}
				\eta_\Theta^{k}=&\|\left(f_w(s(\Theta_{h}^{k}))\k{\lambda_{t}}(s(\Theta_{h}^{k}))-f_w(s(\Theta_{h}^{k-1}))\k{\lambda_{t}}(s(\Theta_{h}^{k-1}))\right)\del P_{h}^{k}\|,\\
				\eta_s^{k}=&\|s'(\Theta_{h}^{k})^{-\frac{1}{2}}(L(\Theta_{h}^{k}-\Theta_{h}^{k-1})-(s(\Theta_{h}^{k})-s(\Theta_{h}^{k-1}))\|,\\
				\eta_{\lambda_{t}}^{k}=&\|(\k{\lambda_{t}}( s(\Theta_{h}^{k})))^{-\frac{1}{2}}\left(\k{\lambda_{t}}( s(\Theta_{h}^{k}))-\k{\lambda_{t}} (s(\Theta_{h}^{k-1}))\right)\del P_{h}^{k}\|.
			\end{aligned}
		\end{equation*}
	\end{lemma}
	\begin{proof}
		\noindent Step 1. 
		First observe from the defintion of the bilinear form for Newton's method \eqref{eq: two phase flow bilinear form Newton} and Newton's method expressed through the residual \eqref{eq: twophase Newton} implies that Newtons method \eqref{eq: newtons method}  with $q_h=\delta\Theta_{h}^{k+1}, r_h= \delta P_{h}^{k+1}$ satisfies
		\begin{equation*}
			\begin{aligned}
				\begin{split}
					\langle s'(\Theta_{h}^{k})(\delta\Theta_{h}^{k+1},q_h\rangle +\tau\langle\del\delta\Theta_{h}^{k+1},\del q_h\rangle\\
					+\tau\langle (f_w\circ s)'(\Theta_{h}^{k})\k{\lambda_{t}}(s(\Theta_{h}^{k}))\del (P_{h}^{k})\delta \Theta_{h}^{k+1},\del q_h\rangle+
					\tau\langle f_w(s(\Theta_{h}^{k}))\k(\lambda\circ s)'(\Theta_{h}^{k}))\del (P_{h}^{k})\delta \Theta_{h}^{k+1},\del q_h\rangle\\
					+\tau\langle f_w(s(\Theta_{h}^{k}))\k{\lambda_{t}}(s(\Theta_{h}^{k}))\del \delta P_{h}^{k+1},\del q_h\rangle\\
					=\tau\langle f_1(s),q_h\rangle-\langle s(\Theta_{h}^{k})-s(\Theta_{h}^{n-1}),q_h\rangle-\tau\langle\del\Theta_{h}^{k},\del q_h\rangle-\tau\langle f_w(s(\Theta_{h}^{k}))\k{\lambda_{t}}(s(\Theta_{h}^{k}))\del P_{h}^{k},\del q_h\rangle,
				\end{split}\\
				\tau\langle \k(\lambda\circ s)'(\Theta_{h}^{k})\del (P_{h}^{k})\delta \Theta_{h}^{k+1},\del r_h\rangle+\tau\langle \k{\lambda_{t}}(s(\Theta_{h}^{k}))\del \delta P_{h}^{k+1},\del r_h\rangle=\tau\langle f_2(s),r_h\rangle-\tau\langle \k{\lambda_{t}}(s(\Theta_{h}^{k}))\del P_{h}^{k},\del r_h\rangle.
			\end{aligned}
		\end{equation*}
		Observe that the left-hand side contains the norm defintion \eqref{eq: twophase newton norm} and the non-symmetric part of the bilinear form for Newton similar to \eqref{eq: framework step 1}. Therefore, we can express the norm as follows
		\begin{align*}
			&\norm{\delta\Theta_{h}^{k+1},\delta P_{h}^{k+1}}_{N,(\Theta_{h}^{k},P_{h}^{k})}^{2}= \left(\int_{\Omega}s'(\Theta_{h}^{k})|\delta\Theta_{h}^{k+1}|^{2}+\tau|\del\delta\Theta_{h}^{k+1}|^{2}+|(\k{\lambda_{t}}(s(\Theta_{h}^{k})))^{\frac{1}{2}}\del \delta P_{h}^{k+1}|^{2}\right)\\
			&=\,\underbrace{\tau\langle f_1(s),q_h\rangle-\langle s(\Theta_{h}^{k})-s(\Theta_{h}^{n-1}),\delta\Theta_{h}^{k+1}\rangle-\tau\langle\del\Theta_{h}^{k},\del \delta\Theta_{h}^{k+1}\rangle-\tau\langle f_w(s(\Theta_{h}^{k}))\k{\lambda_{t}}(s(\Theta_{h}^{k}))\del P_{h}^{k},\del \delta\Theta_{h}^{k+1}\rangle}_{:=\Gamma_1}
			\\
			&+\underbrace{\tau\langle f_2(s),\delta P_{h}^{k+1}\rangle-\tau\langle \k{\lambda_{t}}(s(\Theta_{h}^{k}))\del P_{h}^{k},\del \delta P_{h}^{k+1}\rangle}_{\Gamma_2}
			\\
			&-\underbrace{\tau\langle (f_w\circ s)'(\Theta_{h}^{k})\k{\lambda_{t}}(s(\Theta_{h}^{k}))\del (P_{h}^{k}) \delta \Theta_{h}^{k+1},\del \delta\Theta_{h}^{k+1}\rangle-
				\tau\langle f_w(s(\Theta_{h}^{k}))\k(\lambda\circ s)'(\Theta_{h}^{k})\del(P_{h}^{k})  \delta \Theta_{h}^{k+1},\del \delta\Theta_{h}^{k+1}\rangle}_{:=\Gamma_3}
			\\
			&-\underbrace{\tau\langle \k(\lambda\circ s)'(\Theta_{h}^{k})\del (P_{h}^{k})\delta \Theta_{h}^{k+1},\del \delta P_{h}^{k+1}\rangle}_{\Gamma_{4}}-\underbrace{\tau\langle f_w(s(\Theta_{h}^{k}))\k{\lambda_{t}}(s(\Theta_{h}^{k}))\del \delta P_{h}^{k+1},\del \delta \Theta_{h}^{k+1}\rangle}_{\Gamma_{5}}.
		\end{align*}
		\noindent Step 2.  We consider both $\Gamma_1$ and $\Gamma_2$ which corresponds to only considering the residual, see \Cref{def: twophase residual}. Then by using \eqref{eq: twophase Lscheme}, i.e. adding $\mathcal{B}_{L,\Theta_{h}^{k}}((\delta\Theta_{h}^{k},\delta P_{h}^{k}),(\delta\Theta_{h}^{k+1},\delta P_{h}^{k+1}))+\langle\mathcal{R}(\Theta_{h}^{k-1},P_{h}^{k-1}),(\delta\Theta_{h}^{k+1},\delta P_{h}^{k+1})\rangle$, we see that 
		\begin{align*}
			\Gamma_1+\Gamma_2 =& \langle L\delta\Theta_{h}^{k},\delta\Theta_{h}^{k+1}\rangle- \langle s(\Theta_{h}^{k})-s(\Theta_{h}^{k-1}),\delta\Theta_{h}^{k+1}\rangle\\
			&- \tau\langle \left(f_w(s(\Theta_{h}^{k}))\k{\lambda_{t}}(s(\Theta_{h}^{k}))-f_w(s(\Theta_{h}^{k-1}))\k{\lambda_{t}}(s(\Theta_{h}^{k-1}))\right)\del P_{h}^{k},\del \delta\Theta_{h}^{k+1}\rangle\\
			&-\tau\langle \left(\k{\lambda_{t}}( s(\Theta_{h}^{k}))-\k{\lambda_{t}} (s(\Theta_{h}^{k-1}))\right)\del P_{h}^{k},\del\delta P_{h}^{k+1}\rangle.
		\end{align*}
		\noindent Step 3. We now aim to recover the iteration-dependent norm. By division and multiplication with the factors $s'(\Theta^{k})$ and $(\k{\lambda_{t}}( s(\Theta_{h}^{k})))$ along with the Cauchy-Scwarz inequality we have that
		\begin{align*}
			\Gamma_1+\Gamma_2\leq&\, \|s'(\Theta^{k})^{-\frac{1}{2}}(L(\Theta_{h}^{k}-\Theta_{h}^{k-1})-(s(\Theta_{h}^{k})-s(\Theta_{h}^{k-1}))\|\|s'(\Theta)^{\frac{1}{2}}\delta\Theta_{h}^{k+1}\|\\
			&+\tau\|\left(f_w(s(\Theta_{h}^{k}))\k{\lambda_{t}}(s(\Theta_{h}^{k}))-f_w(s(\Theta_{h}^{k-1}))\k{\lambda_{t}}(s(\Theta_{h}^{k-1}))\right)\del P_{h}^{k}\|\|\del\delta\Theta_{h}^{k+1}\|\\
			&+\tau\|(\k{\lambda_{t}}( s(\Theta_{h}^{k})))^{-\frac{1}{2}}\left(\k{\lambda_{t}}( s(\Theta_{h}^{k}))-\k{\lambda_{t}} (s(\Theta_{h}^{k-1}))\right)\del P_{h}^{k}\|\|(\k{\lambda_{t}}( s(\Theta_{h}^{k})))^{\frac{1}{2}}\del\delta P_{h}^{k+1}\|.
		\end{align*}
		Further, by using the Cauchy-Schwarz inequality we obtain
		\begin{align}\label{eq: two phase proof step 3}
			\begin{split}
				\Gamma_1+\Gamma_2\leq&\, \eta_s\|(s'(\Theta_{h}^{k}))^{\frac{1}{2}}\delta\Theta_{h}^{k+1}\|+\sqrt{\tau}\eta_\Theta\sqrt{\tau}\|\del\delta\Theta_{h}^{k+1}\|+\sqrt{\tau}\eta_{\lambda_{t}}\sqrt{\tau}\|\k{\lambda_{t}}( s(\Theta_{h}^{k}))^{\frac{1}{2}}\del\delta P_{h}^{k+1}\| \\
				\leq& ([\eta_s]^{2}+\tau[\eta_\Theta]^{2}+\t[\eta_{\lambda_{t}}]^{2})^{\frac{1}{2}}\\&\cdot\left(\|(s'(\Theta_{h}^{k}))^{\frac{1}{2}}\delta\Theta_{h}^{k+1}\|^{2}+\tau\|\del\delta\Theta_{h}^{k+1}\|^{2}+\t\|(\k{\lambda_{t}}( s(\Theta_{h}^{k})))^{\frac{1}{2}}\del\delta P_{h}^{k+1}\|^{2}\right)\\
				\leq& ([\eta_s^{k}]^{2}+\tau[\eta_\Theta^{k}]^{2}+\t[\eta_{\lambda_{t}}^{k}]^{2})^{\frac{1}{2}}\norm{\delta\Theta_{h}^{k+1},\delta P_{h}^{k+1}}_{N,(\Theta_{h}^{k},P_{h}^{k})}.
			\end{split}
		\end{align}
		\noindent Step 4. Finally, we look at the non-symmetric part of the bilinear form, i.e. the terms $\Gamma_3$, $\Gamma_4$, and $\Gamma_5$. First we look at $\Gamma_3$, and simplify the notation by denoting $\bm{a}^{k}=\left((f_w\circ s)'(\Theta_{h}^{k})\k{\lambda_{t}}(s(\Theta_{h}^{k}))+f_w(s(\Theta_{h}^{k}))\k(\lambda\circ s)'(\Theta_{h}^{k})\right)\del P_{h}^{k}$. Then by using \Cref{ass: 1} we estimate that
		\begin{align*}
			\Gamma_3 =&\, \tau\left\langle \bm{a}^{k}\delta\Theta_{h}^{k+1}, \del \delta\Theta_{h}^{k+1}\right\rangle\\
			\leq&\,\left(\tau\int_{\Omega}|\bm{a}^{k}|^{2}(\delta\Theta_{h}^{k+1})^{2}\right)^{\frac{1}{2}}\left(\tau\int_{\Omega}|\del\delta\Theta_{h}^{k+1}|^{2}\right)^{\frac{1}{2}}\\
			\leq&\, C_1^{k}\left(\int_{\Omega}\frac{1}{4}s'(\Theta_{h}^{k})(\delta\Theta_{h}^{k+1})^{2}\right)^{\frac{1}{2}}\left(\tau\int_{\Omega}|\del\delta\Theta_{h}^{k+1}|^{2}\right)^{\frac{1}{2}}
			\\
			\leq&\, \frac{C_1^{k}}{2}\int_\Omega \left(\frac{1}{2} s'(\Theta_{h}^{k})(\delta\Theta_{h}^{k+1})^{2}+\frac{1}{2}\tau|\del\delta\Theta_{h}^{k+1}|^{2}	\right).		
		\end{align*}
		Next we need to estimate $\Gamma_{4}$, to simplfy the expression we introduce $\bm{b}^{k}=(\k{\lambda_{t}}(s(\Theta_{h}^{k})))^{-\frac{1}{2}}(\lambda\circ s)'(\Theta_{h}^{k})\del P_{h}^{k}$. Similar to $\Gamma_3$ by using \Cref{ass: 2} we obtain 
		\begin{align*}
			\Gamma_4 =&\, \tau\left\langle \bm{b}^{k}\delta \Theta_{h}^{k+1}, \del \delta P_{h}^{k+1}\right\rangle\\
			\leq&\,\left(\tau\int_{\Omega}|(\k{\lambda_{t}}(s(\Theta_{h}^{k})))^{-\frac{1}{2}}\bm{b}^{k}|^{2}(\delta\Theta_{h}^{k+1})^{2}\right)^{\frac{1}{2}}\left(\t\int_{\Omega}|(\k{\lambda_{t}}(s(\Theta_{h}^{k})))^{\frac{1}{2}}\del\delta P_{h}^{k+1}|^{2}\right)^{\frac{1}{2}}\\
			\leq&\, C_2^{k}\left(\int_{\Omega}\frac{1}{4}s'(\Theta_{h}^{k})(\delta\Theta_{h}^{k+1})^{2}\right)^{\frac{1}{2}}\left(\t\int_{\Omega}|(\k{\lambda_{t}}(s(\Theta_{h}^{k})))^{\frac{1}{2}}\del\delta P_{h}^{k+1}|^{2}\right)^{\frac{1}{2}}
			\\
			\leq&\, \frac{C_2^{k}}{2}\int_\Omega \left(\frac{1}{2}s'(\Theta_{h}^{k})(\delta\Theta_{h}^{k+1})^{2}+\frac{1}{2}\t|(\k{\lambda_{t}}(s(\Theta_{h}^{k})))^{\frac{1}{2}}\del\delta P_{h}^{k+1}|^{2}	\right).		
		\end{align*}
		Finally, $\Gamma_5$ can be bounded in a similar way to the previous ones using the assumption that the fractional flow function satisfies $|f_w|\leq C_{\!_{f_w}}<1$. We then get that
		\begin{align*}
			\Gamma_5\leq \, \frac{2C_{\!_{f_w}}}{2}\int_\Omega \left(\frac{1}{2}\tau|\del\delta\Theta_{h}^{k+1}|^{2}+\frac{1}{2}\t|(\k{\lambda_{t}}(s(\Theta_{h}^{k})))^{\frac{1}{2}}\del\delta P_{h}^{k+1}|^{2}	\right).
		\end{align*}
		Now combining the estimates for $\Gamma_3$, $\Gamma_4$, and $\Gamma_5$ yields
		\begin{align*}
			\Gamma_3+\Gamma_4+\Gamma_5 \leq \frac{\max(C_1^{k},C_2^{k},2C_{\!_{f_w}})}{2}\norm{\delta\Theta_{h}^{k+1},\delta P_{h}^{k+1}}_{N,(\Theta_{h}^{k},P_{h}^{k})}^{2},
		\end{align*}
		from which we can conclude that the estimate holds.
	\end{proof}
	
	\begin{lemma}[Newton to Newton estimator]
		Let Assumptions \ref{as:auxiliary_functions},\ref{ass: 1} and \ref{ass: 2} hold, and the fractional flow function be bounded by $C_{\!_{f_w}}<1$. Let $\{\Theta_{h}^{k},P_{h}^{k}\}$ be a sequence of iterates generated using Newton's method \eqref{eq: twophase Newton}. Then, if $\hat{\Theta}^{k+1},\hat{P}^{k+1}$ are also computed using Newton's method, the incremental error satisfies
		\begin{equation}
			\norm{(\hat{\Theta}^{k+1}-\Theta_{h}^{k},\hat{P}^{k+1}-P_{h}^{k})}_{N,(\Theta_{h}^{k},P_{h}^{k})}\leq \eta_{\!_{2\to 2}}^{k},
		\end{equation}
		where
		\begin{align}\label{eq: twophase est 2to2}
			\eta_{\!_{2\to 2}}^{k}=\frac{2}{2-\max(C_1^{k},C_2^{k},2C_{\!_{f_w}})}\left([\eta_s^{k}]^{2}+\tau[\eta_\Theta^{k}]^{2}+\t[\eta_{\lambda_{t}}^{k}]^{2}\right)^{\frac{1}{2}},
		\end{align}
		with
		\begin{equation*}
			\begin{aligned}
				\eta_\Theta^{k}=&\left\|\begin{split}&\left(f_w(s(\Theta_{h}^{k}))\k{\lambda_{t}}(s(\Theta_{h}^{k}))-f_w(s(\Theta_{h}^{k-1}))\k{\lambda_{t}}(s(\Theta_{h}^{k-1}))\right)\del P_{h}^{k}\\
					&{}-\left((f_w\circ s)'(\Theta_{h}^{k-1})\k{\lambda_{t}}(s(\Theta_{h}^{k-1}))+f_w(s(\Theta_{h}^{k-1}))\k(\lambda\circ s)'(\Theta_{h}^{k-1})\del P_{h}^{k-1}\delta\Theta_{h}^{k}\right)\end{split}\right\|,\\
				\eta_s^{k}=&\|s'(\Theta_{h}^{k})^{-\frac{1}{2}}(s'(\Theta_{h}^{k-1})(\Theta_{h}^{k}-\Theta_{h}^{k-1})-(s(\Theta_{h}^{k})-s(\Theta_{h}^{k-1}))\|,\\
				\eta_{\lambda_{t}}^{k}=&\|(\k{\lambda_{t}}( s(\Theta_{h}^{k})))^{-\frac{1}{2}}\left(\left(\k{\lambda_{t}}( s(\Theta_{h}^{k}))-\k{\lambda_{t}} (s(\Theta_{h}^{k-1}))\right)\del P_{h}^{k}-(\lambda\circ s)'(\Theta_{h}^{k-1})\del P_{h}^{k-1}\delta\Theta_{h}^{k}\right)\|.
			\end{aligned}
		\end{equation*}
	\end{lemma}
	\begin{proof}
		The proof is similar to the proof of the previous lemma and is therefore omitted.
	\end{proof}
	
	\begin{lemma}[L-scheme to L-scheme estimator]
		Let \Cref{as:auxiliary_functions} hold and the fractional flow function be bounded by $C_{\!_{f_w}}<1$. Let $\{\Theta_{h}^{k},P_{h}^{k}\}$ be a sequence of iterates generated using the L-scheme \eqref{eq: twophase Lscheme}. Then, if $\hat{\Theta}^{k+1},\hat{P}^{k+1}$ are also computed using the L-scheme, the incremental error satisfies
		\begin{equation}
			\norm{(\hat{\Theta}^{k+1}-\Theta_{h}^{k},\hat{P}^{k+1}-P_{h}^{k})}_{L,(\Theta_{h}^{k},P_{h}^{k})}\leq \eta_{\!_{1\to 1}}
		\end{equation}
		where
		\begin{align}\label{eq: twophase est 1to1}
			\eta_{\!_{1\to 1}}=\frac{2}{2-2C_{\!_{f_w}}}\left([\eta_s]^{2}+\tau[\eta_\Theta]^{2}+\t[\eta_{\lambda_{t}}]^{2}\right)^{\frac{1}{2}},
		\end{align}
		with
		\begin{equation*}
			\begin{aligned}
				\eta_\Theta^{k}=&\|\left(f_w(s(\Theta_{h}^{k}))\k{\lambda_{t}}(s(\Theta_{h}^{k}))-f_w(s(\Theta_{h}^{k-1}))\k{\lambda_{t}}(s(\Theta_{h}^{k-1}))\right)\del P_{h}^{k}\|,\\
				\eta_s^{k}=&\|L^{-\frac{1}{2}}(L(\Theta_{h}^{k}-\Theta_{h}^{k-1})-(s(\Theta_{h}^{k})-s(\Theta_{h}^{k-1}))\|,\\
				\eta_{\lambda_{t}}^{k}=&\|(\k{\lambda_{t}}( s(\Theta_{h}^{k})))^{-\frac{1}{2}}\left(\left(\k{\lambda_{t}}( s(\Theta_{h}^{k}))-\k{\lambda_{t}} (s(\Theta_{h}^{k-1}))\right)\del P_{h}^{k}\right)\|.
			\end{aligned}
		\end{equation*}
	\end{lemma}
	\begin{proof}
		The proof again is similar to \Cref{lem: twophase L to N}.
	\end{proof}
	Note that the constants $C_1^{k},C_{2}^{k}$ are only needed in the estimators involving the Newton method. This reflects that Newton's method is not expected to converge when $s'(\cdot)$ blows up, in which case no switch to Newton's method should be pursued.  
	\begin{remark}[Extension to mixed formulations]
		To extend the estimators above to a mixed formulation, the steps are very similar. The main difference is that there are now additional non-symmetric terms in the bilinear form. These can be controlled by choosing the test function in the mixed finite element equivalent of \eqref{eq: two phase flow bilinear form Newton} such that the coupling terms in the bilinear form cancel. It does lead to additional terms in the residual, but here we can recover the iteration-dependent norm. 
	\end{remark}
	\begin{remark}
		Although the analysis within this section and in \Cref{sec: surfactant transport} focuses on Newton's method and the L-scheme, the techniques employed here can be used to derive switching estimators from other methods to Newton's method. For example, from the modified L-scheme \cite{mitra_modified_2019} or modified Picard \cite{celia_general_1990} to Newton's method. 
	\end{remark}
	
	\subsection{Adaptive algorithms}
	By considering the estimates in \Cref{sec: twophase estimators} with the estimators  $\eta_{\!_{1\to 2}}$, $\eta_{\!_{2\to 2}}$  and $\eta_{\!_{1\to 1}}$  defined in \eqref{eq: twophase est 1to2}, \eqref{eq: twophase est 2to2} and \eqref{eq: twophase est 1to1}. We propose two adaptive iterative algorithms, a switching algorithm between the L-scheme and Newton's method, and an adaptive tuning of the stabilization parameter in the L-scheme. Also, recall the additional computational considerations in \Cref{sec: comput considerations}.   The algorithms are then:
	\begin{algorithm}[H]
		\caption{L-scheme/Newton \it{a-posteriori} switching}\label{alg: a-posteriori 1}
		\begin{algorithmic}
			\Require $\Theta^{n,0}, P^{n,0}$ as initial guess. 
			\Ensure Scheme=\fbox{L-scheme}
			\For{k=1,2,..}
			\If{Scheme=\fbox{L-scheme}} 
			\State
			Compute iterate using  \fbox{L-scheme}, i.e. \eqref{eq: twophase Lscheme}        
			\If{$\eta_{\!_{1\to 2}}^k\leq \eta^k_{\!_{\,inc,1}}$}
			\State Set Scheme=\fbox{Newton}
			\EndIf
			\Else 
			\State \  Compute iterate using \fbox{Newton}, i.e. \eqref{eq: twophase Newton}
			\If{$\eta_{\!_{2\to 2}}^k> \eta^k_{\!_{\,inc,2}}$}
			\State Set Scheme=\fbox{L-scheme}
			\EndIf
			\EndIf
			\EndFor

		\end{algorithmic}
	\end{algorithm}
	
	\begin{algorithm}[H]
		\caption{L-scheme adaptive stabilization parameter}\label{alg: a-posteriori 3}
		\begin{algorithmic}
			\Require $\Theta^{n,0}, P^{n,0}$ as initial guess. 
			\Ensure Scheme=\fbox{L-scheme}, $C_{\!_{1\to 1}}=\sqrt{2}$
			\For{k=1,2,..}
			\State
			Compute iterate using  \fbox{L-scheme}, i.e. \eqref{eq: twophase Lscheme}        
			\If{$\eta_{\!_{inc,1}}^{k}\geq \eta_{\!_{1\to 1}}^k\geq  0.8\eta_{\!_{\,inc,1}}^{k}$}
			\State Decrease size of $L=0.8L$
			
			\EndIf
			\If{$\eta_{\!_{1\to 1}}^k> \eta_{\!_{\,inc,1}}^{k}$}
			\State Increase size of $L=C_{\!_{1\to 1}}L$
			\EndIf
			\EndFor

		\end{algorithmic}
	\end{algorithm}
	
	\subsection{Numerical example}
	Here, we perform a numerical experiment to test the performance of the proposed algorithms. The adaptive L-scheme (\Cref{alg: a-posteriori 3}) will be referred to as the $L-A$ scheme, and the algorithm for switching between Newton and the L-scheme (\Cref{alg: a-posteriori 1}) will be referred to as the $L-N$ scheme. Since the L-scheme for this problem is known to have performance issues for small H\"older exponents, we consider a problem similar to \cite{radu_robust_2018} where we decrease the H\"older exponent. We use the following constitutive equations 
	\begin{equation}\label{eq: two phase s expression}
		\lambda_w=s^{\gamma},\, \lambda_n=1-s^{\gamma},\,s(\Theta)=\Theta^{\gamma},
	\end{equation}
	and let $k=10^{-5}$. This choice leads to the relations
	\begin{equation}
		\lambda_t(s)=s^{\gamma}+(1-s)^{\gamma},\quad f_w(s)=\frac{s^{\gamma}}{s^{\gamma}+(1-s)^{\gamma}}.
	\end{equation}
	Further, the source terms are set to zero.
	Let $\Omega=(0,1)^{2}$ be the unit square. We consider an initial saturation of 0.2 in the entire domain, except for in two regions. Near the bottom of the domain the elements bordering $y=0$ we have an initial saturation of $s=0.6$ and within the circle $\{(x-0.5)^{2}+(y-0.5)^{2}\leq 0.1\}$ is equal to zero. This is to ensure that we have a degeneracy at the beginning, where the estimators are only computed outside of this region. Note that $s(\Theta)$ is only H\"older continuous in this region. No-flow boundary conditions are used for the saturation. For the global pressure, we consider $P=1$ at the bottom and $P=0$ at the top, along with homogeneous Neumann conditions at the left and right boundaries of the domain. This means that there is a flow from the bottom to the top.

	Recall that the estimators \eqref{eq: twophase est 1to2}, \eqref{eq: twophase est 2to2}, and \eqref{eq: twophase est 1to1} all assume that the fractional flow function is strictly less than 1, meaning there should always be some residual non-wetting fluid. This is not satisfied by the initial condition in this example. In addition, the initial condition and expression of $s(\cdot)$ in ~\eqref{eq: two phase s expression} mean that the analytical evaluation of the derivatives in Newton's method~\eqref{eq: two phase flow bilinear form Newton} yields \textit{NaN} and Newton will not converge. Therefore, we also include a comparison against a modified Newton method, which locally drops the derivatives in the case of phase disappearance to bypass the \textit{NaN} evaluation, mimicking a local Picard method, which we label Newton-$mod$. No such modification is done when switching from the L-scheme to Newton's method. For the L-scheme, we consider two different stabilization parameters $L_1=1$ and $L_2=10$. Since the analysis presented for the H\"older continuous case \cite{radu_robust_2018} only guarantees convergence for $L$ large enough depending on the H\"older exponent. We stop the iteration process when 
	\begin{equation}
		\|\Theta_h^{k}-\Theta_h^{k-1}\|\leq 10^{-6},\mbox{ and }\|P_h^{k}-P_h^{k-1}\|\leq 10^{-6}.
	\end{equation}
	\subsubsection{Comparison of convergence properties}
	
	Here, we discuss the performance of the proposed algorithms and compare them to the L-scheme, Newton's method, and Newton-$mod$. The number of iterations for the different methods for $\gamma\in\{0.5,0.9\}$ is presented in \Cref{tab: twophase}. The L-scheme for the smaller stabilization parameter only converges for $\gamma\geq 0.7$, whereas for the larger value $L_2$ it converges for all $\gamma$ considered. Both require more iterations to reach convergence than all other methods, with $L_1$ using fewer than $L_2$ when it converges. Newton's method as expected does not converge and the Newton-$mod$ is able to bypass the difficulties in the initial condition but only converges for $\gamma\geq 0.6$. When it converges, it uses very few iterations compared to the other methods, the only scheme that uses fewer iterations is the $L_1-N$-scheme. However, the $L_1-N$-scheme does not converge for all $\gamma$. The $L_2-N$-scheme is able to compete with the Newton-$mod$ method and, in addition, still performs like a higher-order method when $\gamma=0.5$. Here, it uses less than half the number of iterations compared with the second fastest algorithm being the $L_1-A$-scheme. Both the adaptive L-schemes converge for all $\gamma$. It is worth noting that it appears that starting with a small stabilization parameter works better in terms of the number of iterations for the adaptive schemes, as the decrease in $L$ is not aggressive enough. The $L_2-A$ scheme uses fewer iterations for smaller $\gamma$, meaning that, depending on the problem, smaller stabilization parameters that still give convergence are likely to be faster.
	\renewcommand{\arraystretch}{1.3} 
	\begin{table}[H]
		\centering
		\begin{tabular}{c|ccccc}
			\hline
			\multicolumn{1}{c|}{$\gamma$} & 
			\multicolumn{1}{c}{0.9} & 
			\multicolumn{1}{c}{0.8} & 
			\multicolumn{1}{c}{0.7} & 
			\multicolumn{1}{c}{0.6} & 
			\multicolumn{1}{c}{0.5} \\
			\hline
			Scheme  &
			Avg. Itr. & 
			Avg. Itr. &  
			Avg. Itr. &
			Avg. Itr. & 
			Avg. Itr.  \\
			\hline
			
			$L_1$ &  3.7& 5.1  & 9.3  & - & -  \\
			
			$L_2$ &  20.7  & 20.2 & 18.6 & 15.6 &  11.8  \\
			
			Newton & -  & - & -  & -  & -  \\
			Newton-$mod$ & 3.3  & 3.3 & 3.3  & 3.5  & -  \\
			
			$L_1-A$ &  3.7  & 5.1  & 6.8  & 5 &  7.6 \\
			
			$L_2-A$ &  13.3  & 12.8  & 13.4  & 13.3  & 11.8 \\
			
			$L_1-N$ &  3.0(1)  & 3.1(1)  & 3.2(1)  & - & -\\
			
			$L_2-N$  & 3.3(1)  & 3.4(1) &  3.4(1)  & 3.5(1)  &  3.6(1) \\
			\hline
		\end{tabular}
		\captionsetup{width=\textwidth}
		\caption{Test case: two-phase flow - Average and total number of iterations at $T=1$ with $\tau=0.1$ for different iterative methods under varying H\"older exponent $\gamma$. By $-$ we denote divergence. Number in parentheses $(\cdot)$ represents the number of L-scheme iterations before the adaptive algorithm switches to Newton's method.}
		\label{tab: twophase}
	\end{table}

	\subsubsection{Switching characteristics}
	
	Finally, we look at the adaptive behavior of switching between the L-scheme and Newton's method, and at the adaptive update of the stabilization parameter $L$. In \Cref{fig: twophase L adaptive}, the ratio between the estimator and the incremental error for both $L-A$ schemes is displayed when $\gamma=0.7$. The ratio is used to tune $L$. For $L_1-A$, the decrease happens quickly, immediately followed by an increase in $L$ due to the decrease being too large, indicating that it is close to the lower limit. When $L$ starts off being large for the $L_2-A$ scheme, it takes 10 iterations before the first decrease, and then is closely followed by a second decrease. Still, for both, we see that the ratio eventually borders right under 0.8, and nothing else happens until convergence is reached. This behavior is also the reason for $L_1-A$ outperforming $L_2-A$ for all $\gamma$, but in \Cref{tab: twophase} for $\gamma=0.5$ the difference becomes smaller. 
	\begin{figure}[H]
		\centering
		\includegraphics[width=0.5\textwidth]{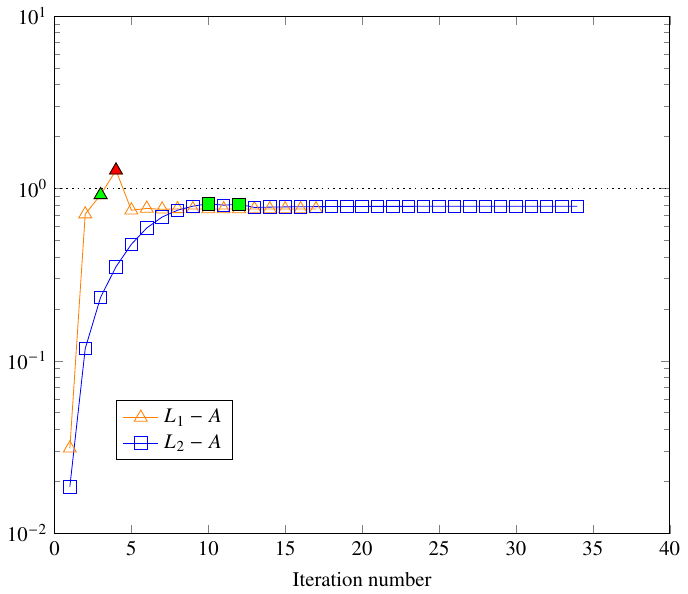}
		\caption{Test case: two-phase flow - The ratio between $\eta_{\!_{1\to 1}}/\eta_{\!_{inc,1}}$ at each iteration, i.e. the criteria for decreasing or increasing $L$ in the adaptive algorithm for the first time step when $\gamma=0.7$.  Green indicates a decrease and red an increase of the stabilization parameter. The dashed line is $C_{\rm tol}=1$.}
		\label{fig: twophase L adaptive}
	\end{figure}
	The switching indicators for the $L-N$ scheme are plotted in \Cref{fig: twophase evolution} for $\gamma=0.5$, along with the efficiency indices for $\eta_{\!_{2\to 2}}$. The switch happens after the first iteration, and Newton's method then converges. Note that the efficiency index is above 1,  despite not computing the constants $C_1^{k}$ and $C_{2}^{k}$, which are needed to have a guaranteed upper bound. Nevertheless, the estimators appear to accurately indicate when a switch from the L-scheme to Newton's method can take place.  
	\begin{figure}[H]
		\centering
		\includegraphics[width=0.5\textwidth]{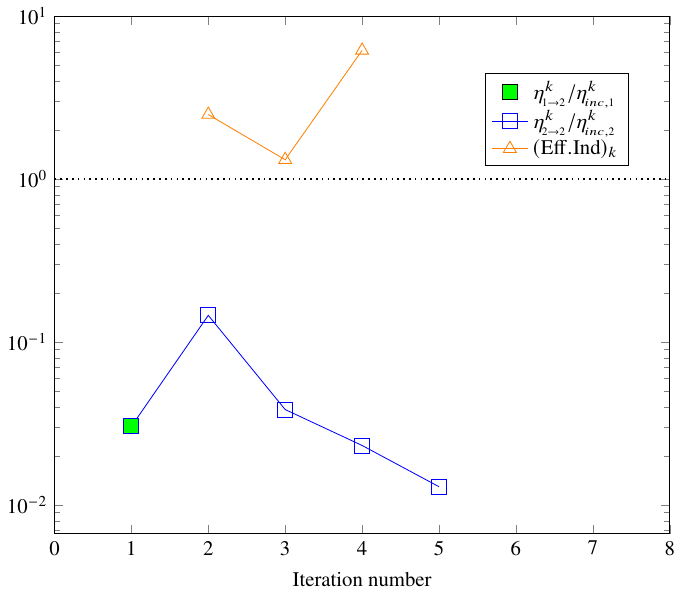}
		\caption{Test case: two-phase flow - Evolution of switching indicators for the $L_2-N$ scheme for $h=\sqrt{2}/40$ and $\tau=0.1$ when $\gamma=0.5$. The effectivity indices \eqref{eq: eff ind} corresponding to the Newton iterations ($\eta_{\!_{2\to 2}}$) are also plotted. The dashed line is $C_{\rm tol}=1$.}
		\label{fig: twophase evolution}
	\end{figure}
	\section{Surfactant transport in porous media}\label{sec: surfactant transport}

	In this section, we consider the transport of a surfactant in a variably saturated medium. For the water flow, we consider Richards' equation, and it will be fully coupled with a reaction-diffusion-convection equation for the transport of the surfactant. Precisely we consider the same model as \cite{illiano_iterative_2021}: Find the pressure head ($\psi$) and the surfactant concentration ($c$) such that 
	\begin{subequations}\label{eq: surfactant}
		\begin{align}
			\p_{t}\theta(\psi,c)-\div (K(\theta(\psi,c))\del(\psi+z))&=f_3,\quad&&\mbox{in }\Omega\times[0,T],\\
			\p_{t}(\theta(\psi,c)c)-\div (D\del c-\bm{u}_{w}c)&=f_4,\quad&&\mbox{in }\Omega\times[0,T],
		\end{align}
	\end{subequations}
	where $\theta$ is the water content, $K$ is the hydraulic conductivity, $D>0$ the diffusion coefficient and $z$ the height against gravitational direction. Lastly $\bm{u}_w:=-K(\theta(\psi,c))\del(\psi+z)$ is the water flux and $f_3,f_4$ are sink/source terms. 
	\begin{assumption}\label{as:auxiliary_functions2}
		For the material properties $\theta(\cdot)$ and $K(\cdot)$, and source terms $f_3(\cdot),f_4(\cdot)$ in \eqref{eq: surfactant}, the following assumptions are made:
		\begin{itemize}
			\item[(a)] There exists constants $a_{\psi}>0$ and $a_{c}\geq 0$ such that for any $\psi_1,\psi_2\in\R$ and $c_1,c_2\in\R_{+}$  the water content function $\theta$ satisfies 
			\begin{equation*}
				\begin{aligned}
					&\langle\theta(\psi_1,c_1)-\theta(\psi_2,c_2),\psi_1-\psi_2\rangle+\langle c_1\theta(\psi_1,c_1)-c_2\theta(\psi_2,c_2),c_1-c_2\rangle\\&\geq a_{\psi}\|\theta(\psi_1,c_1)-\theta(\psi_2,c_2)\|^{2}+a_{c}\|\psi_1-\psi_2\|^{2}.
				\end{aligned}
			\end{equation*} Additionally, there exists two constants $\infty>\theta^{M},\theta^{m}\geq0$ such that $0\leq \theta^{m}\leq \theta(\psi,c)\leq \theta^{M}<\infty,\,\forall\,\psi,c\in\R$.
			\item[(b)] The hydraulic conductivity $K$ is Lipschitz continuous with respect to both $\psi$ and $c$ and there exists two constants $K^{M},K^{m}>0$ such that $0< K^{m}\leq K\leq K^{M}<\infty$.
			\item[(c)] There exits constants $M_u,M_\psi, M_c\geq 0$ such that $\|\bmu_{w}^{n}\|_{L^{\infty}}\leq M_u, \|\del\psi^{n}\|_{L^{\infty}}\leq M_\psi$ and $\|c^{n}\|_{L^{\infty}}\leq M_c$ for all $n\in\mathbb{N}$.
			\item[(d)] The source functions satisfy $f_3,f_4\in C(0,T;L^2(\Om)).$
		\end{itemize}
	\end{assumption}
	\begin{remark}
		The first assumption \ref{as:auxiliary_functions2}~(a) is a coercivity condition, and in most cases is satisfied if $\theta$ only depends on $\psi$. The assumption about Lipschitz continuity for the hydraulic conductivity in \ref{as:auxiliary_functions2}~(b) is realistic in most situations. The lower bound of $K$ being strictly greater than zero, meaning the soil is never completely dry is also realistic, but is only required in what follows for Newton's method.  4~(c) is valid if the data is sufficiently regular, and is required to show convergence for the L-scheme.
	\end{remark}
	
	For the discretization of \eqref{eq: surfactant}, we will use the backward Euler method in time, and $\P 1$ elements in space. The fully discrete formulation at time $t_n$ is then: Given $\psi_{h}^{n-1},c_{h}^{n-1}\in Q_h$ find $\psi_{h}^{n},c^{n}\in Q_h$ such that
	\begin{subequations}\label{eq: surfactant transport semi-discrete}
		\begin{align}
			\langle \theta(\psi_{h}^{n},c^{n})-\theta(\psi_{h}^{n-1},c_{h}^{n-1}),q_h\rangle +\tau \langle K(\theta(\psi_{h}^{n},c^{n}))\del(\psi_{h}^{n}+z),\del q_h\rangle = \tau\langle f_3^{n},q_h\rangle,\\
			\langle \theta(\psi_{h}^{n},c^{n})c^{n}-\theta(\psi_{h}^{n-1},c_{h}^{n-1})c_{h}^{n-1},r_h\rangle +\tau \langle D\del c^{n}+\bm{u}_w^{n-1}c^{n},\del r_h\rangle = \tau\langle f_4^{n},r_h\rangle,
		\end{align}
	\end{subequations}
	for all $q_h, r_h\in Q_h$.
	Similar to \cite{illiano_iterative_2021}, where computations were run for the water flux $\bmu_w$ evaluated at the present time step and the previous one with marginal differences, the water flux $\bmu_w$ is presented at the previous time step for ease of presentation. In the computations here, the water flux is evaluated at the previous iteration in the iterative linearization.  To solve the sequence of nonlinear problems, we will consider two linearization schemes, the L-scheme, due to its robustness with respect to the mesh size, and Newton's method, due to the higher-order convergence. We mention that there exist many linearization schemes for solving the problem above, including the modified Picard method \cite{celia_general_1990}, the modified L-scheme \cite{mitra_modified_2019}, and using the Picard method to generate a good initial guess for Newton's method \cite{bergamaschi_mixed_1999}. 
	
	Newton's method as a linearization scheme is known to put severe restrictions on the time step size for fine meshes \cite{radu2006}. Therefore, our goal in this section is to adaptively switch between the L-scheme and Newton's method and also develop an adaptive time-stepping algorithm based on \textit{a posteriori} error estimates.
	We introduce the total residual of the weak formulation \eqref{eq: surfactant transport semi-discrete} as
	\begin{definition}[Residual surfactant transport]
		For all $q_h,r_h\in Q_h$, the total residual of \eqref{eq: surfactant transport semi-discrete} is defined by
		\begin{align}
			\begin{split}
				\langle \mathcal{R}(\varphi_1,\varphi_2),(q_h,r_h)\rangle:=
				\,\langle \theta(\varphi_1,\varphi_2)-\theta(\psi_{h}^{n-1},c_{h}^{n-1}),q_h\rangle +\tau \langle K(\theta(\varphi_1,\varphi_2))\del(\varphi_1+z),\del q_h\rangle \\
				\quad\quad\,\,\,+\langle \theta(\varphi_1,\varphi_2)\varphi_2-\theta(\psi_{h}^{n-1},c_{h}^{n-1})c_{h}^{n-1},r_h\rangle +\tau \langle D\del \varphi_2+\bm{u}_w^{n-1}\varphi_2,\del r_h\rangle-\tau\langle f_3^{n},q_h\rangle-\tau\langle f_4^{n},r_h\rangle.
			\end{split}
		\end{align}
		\label{def: surfactant residual}
	\end{definition}
	
	\subsection{Linearization methods}
	
	To solve the sequence of non-linear problems \eqref{eq: surfactant transport semi-discrete}, we have chosen to only consider the L-scheme \cite[Eqs. (21-22)]{illiano_iterative_2021} and Newton's method \cite[Eqs. (16-17)]{illiano_iterative_2021}. 
	First we consider the more robust scheme which is the L-scheme: Given $\psi_{h}^{n-1},c_{h}^{n-1},\psi_{h}^{k},c_{h}^{k}\in Q_h$ find $\psi_{h}^{k+1},c_{h}^{k+1}\in Q_h$ such that
	\begin{subequations}\label{eq: surfactant transport original Lscheme}
		\begin{align}
			\begin{split}
				\langle L_1(\psi_{h}^{k+1}-\psi_{h}^{k}),q_h\rangle+\langle \theta(\psi_{h}^{k},c_{h}^{k})-\theta(\psi_{h}^{n-1},c_{h}^{n-1}),q_h\rangle& \\+\tau \langle K(\theta(\psi_{h}^{k},c_{h}^{k}))\del(\psi_{h}^{k+1}+z),\del q_h\rangle &= \tau\langle f_3^{n},q_h\rangle,
			\end{split}\\
			\begin{split}
				\langle L_2(c_{h}^{k+1}-c_{h}^{k}),q_h\rangle+\langle \theta(\psi_{h}^{k},c_{h}^{k})c_{h}^{k+1}-\theta(\psi_{h}^{n-1},c_{h}^{n-1})c_{h}^{n-1},r_h\rangle&\\ +\tau \langle D\del c_{h}^{k+1}+\bm{u}_w^{n-1}c_{h}^{k+1},\del r_h\rangle &= \tau\langle f_4^{n},r_h\rangle
			\end{split}
		\end{align}
	\end{subequations}
	for all $ q_h, r_h\,\in Q_h$. By \cite[Theorem 1]{illiano_iterative_2021} the convergence of the L-scheme is guaranteed under the \Cref{as:auxiliary_functions2} (a)-(c) if the stabilization parameters are chosen such that
	\begin{equation*}
		L_1\geq \frac{2}{a_{\psi}},\quad\mbox{ and }L_2\geq \frac{2M_c^{2}}{a_{\psi}}+\frac{\theta^{M}}{\theta^{m}}.
	\end{equation*}The \cref{as:auxiliary_functions2}~(b) can even be relaxed to allow $K^{m}\geq 0$. Similarly to before, we introduce the bilinear form for the L-scheme
	\begin{equation}\label{eq: surfactant transport bilinear Lscheme}
		\begin{aligned}
			\mathcal{B}_{L,(\psi_{h}^{k},c_{h}^{k})}((\varphi_1,\varphi_2),(q_h,r_h)):=&\, \langle L_1\varphi_1,q_h\rangle+\tau \langle K(\theta(\psi_{h}^{k},c_{h}^{k}))\del(\varphi_1),\del q_h\rangle \\
			&+\langle L_2\varphi_2,r_h\rangle+\langle \theta(\psi_{h}^{k},c_{h}^{k})\varphi_2,r_h\rangle+\tau \langle D\del \varphi_2+\bm{u}_w^{n-1}\varphi_2,\del r_h\rangle,
		\end{aligned}
	\end{equation}
	for all $q_h,r_h\in Q_h$.
	Then the L-scheme can be expressed as
	\begin{iterative}[L-scheme] For the bilinear form defined in \eqref{eq: surfactant transport bilinear Lscheme} and the residual in \Cref{def: surfactant residual} the L-scheme can be defined as
		\begin{align}\label{eq: surfactant transport Lscheme}
			\mathcal{B}_{L,(\psi_{h}^{k},c_{h}^{k})}((\delta\psi_{h}^{k+1},\delta c_{h}^{k+1}),(q_h,r_h))=-\langle \mathcal{R}(\psi_{h}^{k},c_{h}^{k}),(q_h,r_h)\rangle.
		\end{align}
	\end{iterative}

	\begin{definition}[Iteration dependent norm for the L-scheme] For $\varphi_1,\varphi_2\in H_{0}^{1}(\Omega)$, the iteration-dependent norm for the L-scheme \eqref{eq: surfactant transport Lscheme} is defined by
		\begin{align}
			\begin{split}
				\norm{\varphi_1,\varphi_2}_{L,(\psi_{h}^{k},c_{h}^{k})}:=&\biggl(\int_{\Omega} L_1(\varphi_1)^{2}+L_2(\varphi_2)^{2}+\tau|(K(\theta(\psi_{h}^{k},c_{h}^{k})))^{\frac{1}{2}}\del(\varphi_1)|^{2}\\&\,\quad+\theta(\psi_{h}^{k},c_{h}^{k})(\varphi_2)^{2}+\tau |D^{\frac{1}{2}}\del\varphi_2|^{2}\biggr)^{\frac{1}{2}}.
			\end{split}
		\end{align}
	\end{definition} 
	We consider Newton's method as presented in \cite[Eqs.(16-17)]{illiano_iterative_2021}, as they reported practically no difference between it and the traditional Newton method.  The bilinear form for Newton's method is
	\begin{align}\label{eq: surfactant transport bilinear Newton}
		\begin{split}
			\mathcal{B}_{N,(\psi_{h}^{k},c_{h}^{k})}((\varphi_1,\varphi_2),(q_h,r_h)):=\, \langle \frac{\partial \theta}{\partial \psi}\left(\psi_{h}^{k},c_{h}^{k}\right)\varphi_1,q_h\rangle+\tau \langle K(\theta(\psi_{h}^{k},c_{h}^{k}))\del(\varphi_1),\del q_h\rangle \\
			+\tau \langle K'(\theta(\psi_{h}^{k},c_{h}^{k}))\frac{\partial \theta}{\partial \psi}\left(\psi_{h}^{k},c_{h}^{k}\right)\del(\psi_{h}^{k}+z)\varphi_1,\del q_h\rangle\\
			+\langle \frac{\partial \theta}{\partial c}\left(\psi_{h}^{k},c_{h}^{k}\right)\varphi_2,r_h\rangle+\langle \theta(\psi_{h}^{k},c_{h}^{k})\varphi_2,r_h\rangle+\tau \langle D\del \varphi_2+\bm{u}_w^{n-1}\varphi_2,\del r_h\rangle.
		\end{split}
	\end{align}
	Therefore, we can write Newton's method as
	\begin{iterative}[Newton's method] For the bilinear form defined in \eqref{eq: surfactant transport bilinear Newton} and the residual in \Cref{def: surfactant residual}, Newton's method can be defined as
		\begin{align}\label{eq: surfactant transport Newtons method}
			\mathcal{B}_{N,(\psi_{h}^{k},c_{h}^{k})}((\delta\psi_{h}^{k+1},\delta c_{h}^{k+1}),(q_h,r_h))=-\langle \mathcal{R}(\psi_{h}^{k},c_{h}^{k}),(q_h,r_h)\rangle.
		\end{align}
	\end{iterative}
	
	\begin{definition}[Iteration dependent norm for Newton's method] For $\varphi_1,\varphi_2\in H_{0}^{1}(\Omega)$, the iteration-dependent norm for Newton's method \eqref{eq: surfactant transport Newtons method} is defined by
		\begin{align}
			\begin{split}
				\norm{\varphi_1,\varphi_2}_{N,(\psi_{h}^{k},c_{h}^{k})}:=&\biggl(\int_{\Omega} \frac{\partial \theta}{\partial \psi}\left(\psi_{h}^{k},c_{h}^{k}\right)\varphi_1^{2}+\frac{\partial \theta}{\partial c}\left(\psi_{h}^{k},c_{h}^{k}\right)\varphi_2^{2}+\tau|(K(\theta(\psi_{h}^{k},c_{h}^{k})))^{\frac{1}{2}}\del(\varphi_1)|^{2}\\&\,\quad+\theta(\psi_{h}^{k},c_{h}^{k})(\varphi_2)^{2}+\tau |D^{\frac{1}{2}}\del\varphi_2|^{2}\biggr)^{\frac{1}{2}}.
			\end{split}
		\end{align}
	\end{definition}
	Note that if the \cref{as:auxiliary_functions2}~(b) is relaxed to allow $K^{m}\geq 0$, the above quantity becomes only a semi-norm.

	\subsection{Estimators}\label{sec: surfactant estimators}
	Here we derive \textit{a posteriori} error estimators to predict the incremental error when going from the L-scheme to Newton and an estimator for predicting the failure of Newton's method. To simplify notation, we denote by $K^{k}:=K(\theta(\psi_{h}^{k},c_{h}^{k}))$ and $\theta^{k}:=\theta(\psi_{h}^{k},c_{h}^{k})$.
	In this section, we make a similar assumption to \cite[Assumption 2.]{stokke_adaptive_2023}, that
	\begin{assumption}\label{ass: nondom}
		For a $k\in\mathbb{N}$, there exists a constant $C_{3}^{k}\in [0,2)$ such that
		\begin{equation*}
			\tau \left|K(\theta^{k})^{-\frac{1}{2}}K'(\theta^{k})\frac{\partial \theta^{k}}{\partial \psi}\del(\psi_{h}^{k}+z)\right|^{2}\leq \left(C_{3}^{k}\right)^{2}\frac{\partial\theta^{k}}{\partial \psi},
		\end{equation*}
		almost everywhere in $\Omega$.
	\end{assumption}
	Similarly to the two previous assumptions in \Cref{sec: twophaseflow}, \Cref{ass: nondom} always holds in the degenerate region when $\frac{\p\theta^{k}}{\p\psi}=0$ as $\frac{\p\theta^{k}}{\p\psi}^{2}$ appears on the left-hand side of the inequality. The inequality also holds if the numerical flux is bounded, the time step size is small, and if \Cref{as:auxiliary_functions2} (a) holds as $\theta$ is both strictly increasing and assumed to be Lipschitz continuous. Also note that $C_{3}^{k}$ is computable. 
	
	\begin{lemma}[L-scheme to Newton estimator]\label{lem: surfactant transport Lscheme to newton}
		Let Assumptions \ref{as:auxiliary_functions2} and \ref{ass: nondom} hold. Let $\{\psi_{h}^{k},c_{h}^{k}\}$ be a sequence of iterates generated using the L-scheme \eqref{eq: surfactant transport Lscheme}. Then, if $\hat{\psi}^{k+1},\hat{c}^{k+1}$ are computed using Newton's method \eqref{eq: surfactant transport Newtons method}, then the incremental error satisfies
		\begin{equation}
			\norm{(\hat{\psi}^{k+1}-\psi_{h}^{k},\hat{c}^{k+1}-c_{h}^{k})}_{N,(\psi_{h}^{k},c_{h}^{k})}\leq \eta_{\!_{3\to 4}}^{k},
		\end{equation}
		where 
		\begin{equation}\label{eq: surf est 3to4} 
			\begin{aligned}
				\eta_{\!_{\hypertarget{est:34}{3\to 4}}}^{k}:= \frac{2}{2-C_{3}^{k}}\left(\left[\eta_{\psi}^{k}\right]^{2}+\left[\eta_c^{k}\right]^{2}+\tau\left[\eta_{D}^{k}\right]^{2}+\tau\left[\eta_{K}^{k}\right]^{2}\right)^{\frac{1}{2}}
			\end{aligned}
		\end{equation}
		with 
		\begin{equation*}
			\begin{aligned}
				\eta_{D}^{k}:=&\left\|\left(D\right)^{-\frac{1}{2}}\left(D\del\delta c_{h}^{k}+\bmu_w^{n-1}\delta c_{h}^{k}\right)\right\|,\\
				\eta_K^{k}:=&\left\|\left(K^{k}\right)^{-\frac{1}{2}}\left(\delta K^{k}\right)\del(\psi_{h}^{k}+z)\right\|,\\
				\eta_\psi^{k}:=&\left\|\left(\frac{\partial\theta^{k}}{\partial\psi}\right)^{-\frac{1}{2}}\left(L_1\delta\psi_{h}^{k}-\delta\theta^{k}\right)\right\|,\\
				\eta_c^{k}:=&\left\|\left(\frac{\partial\theta^{k}}{\partial c}\right)^{-\frac{1}{2}}\left(L_2\delta c_{h}^{k}-(\delta\theta^{k})c_{h}^{k}\right)\right\|.
			\end{aligned}
		\end{equation*}
	\end{lemma}
	\begin{proof}\noindent Step 1. We see that the norm can be rewritten in terms of the residual and non-symmetric part of the bilinear form, therefore we have
		\begin{equation*}
			\begin{aligned}
				\norm{\delta\psi_{h}^{k+1},\delta c_{h}^{k+1}}_{N,(\psi_{h}^{k},c_{h}^{k})}^{2}:=&\,\int_{\Omega} \frac{\partial \theta^{k}}{\partial \psi}(\delta\psi_{h}^{k+1})^{2}+\frac{\partial \theta^{k}}{\partial c}(\delta c_{h}^{k+1})^{2}+\tau|(K^{k})^{\frac{1}{2}}\del(\delta\psi_{h}^{k+1})|^{2}\\&+\theta^{k}(\delta c_{h}^{k+1})^{2}+\tau |D^{\frac{1}{2}}\del\delta c_{h}^{k+1}|^{2}\\
				\overset{\eqref{eq: surfactant transport Newtons method}}{=}&\, -\underbrace{\tau \langle K'(\theta^{k})\frac{\partial \theta^{k}}{\partial \psi}\del(\psi_{h}^{k}+z)\delta\psi_{h}^{k+1},\del \delta\psi_{h}^{k+1}\rangle}_{:=\Gamma_1}\\
				&+\underbrace{\tau\langle f_3,\delta\psi_{h}^{k+1}\rangle -\langle\theta^{k}-\theta^{n-1},\delta\psi_{h}^{k+1}\rangle-\tau\langle K^{k}\del(\psi_{h}^{k}+z),\del\delta\psi_{h}^{k+1}\rangle}_{:=\Gamma_2}\\
				&+\underbrace{\tau\langle f_4,\delta c_{h}^{k+1}\rangle-\langle \theta^{k}c_{h}^{k}-\theta^{n-1}c_{h}^{n-1},\delta c_{h}^{k+1}\rangle -\tau\langle D\del c_{h}^{k}+\bm{u}_w^{n-1}c_{h}^{k},\del \delta c_{h}^{k+1}\rangle}_{:=\Gamma_3}.
			\end{aligned}
		\end{equation*}
		\noindent Step 2. and Step 3. Consider $\Gamma_2$ and $\Gamma_3$. By using the fact that the bilinear form evaluated at iteration $k$ is equal to the residual at iteration $k-1$, where both are computed using the L-scheme, we see that
		\begin{equation}
			\begin{aligned}
				\Gamma_2+\Gamma_3=&\, \langle L_1(\delta\psi_{h}^{k}),\delta\psi_{h}^{k+1}\rangle-\langle\theta^{k}-\theta^{k-1},\delta\psi_{h}^{k+1}\rangle-\tau\langle (K^{k}-K^{k-1})\del(\psi_{h}^{k}+z),\del\delta\psi_{h}^{k+1}\rangle\\
				&+\langle L_2(\delta c_{h}^{k}),\delta c_{h}^{k+1}\rangle-\langle (\theta^{k}-\theta^{k-1})c_{h}^{k},\delta c_{h}^{k+1}\rangle \\&-\tau\langle D\del (c_{h}^{k}-c_{h}^{k-1})+\bm{u}_w^{n-1}(c_{h}^{k}-c_{h}^{k-1}),\del \delta c_{h}^{k+1}\rangle.
			\end{aligned}
		\end{equation}
		See that we can follow the same lines as above to obtain a relation with the iteration-dependent norm.
		
		\noindent Step 4.
		We see that $\Gamma_1$ using \Cref{ass: nondom} can be treated similarly to \cite[Eq. (16.b)]{stokke_adaptive_2023}, thus we get
		\begin{align*}
			\Gamma_1 \leq \frac{C_{3}^{k}}{2}\norm{\delta\psi_{h}^{k+1},\delta c_{h}^{k+1}}_{N,(\psi_{h}^{k},c_{h}^{k})}^{2},
		\end{align*}	
		and can conclude that the estimate holds.
		
	\end{proof}
	
	The control of Newton error follows similar steps.
	\begin{lemma}[Newton to Newton estimator]\label{lem: surfactant transport newton to newton}
		Let Assumptions \ref{as:auxiliary_functions2} and \ref{ass: nondom} hold. Let $\{\psi_{h}^{k},c_{h}^{k}\}$ be a sequence of iterates generated using Newton's method \eqref{eq: surfactant transport Newtons method}. Then, if $\hat{\psi}^{k+1},\hat{c}^{k+1}$ are computed using Newton's method \eqref{eq: surfactant transport Newtons method}, then the incremental error satisfies
		\begin{equation}
			\norm{(\hat{\psi}^{k+1}-\psi_{h}^{k},\hat{c}^{k+1}-c_{h}^{k})}_{N,(\psi_{h}^{k},c_{h}^{k})}\leq \eta_{\!_{4\to 4}}^{k},
		\end{equation}
		where 
		\begin{equation}\label{eq: surf est 4to4} 
			\eta_{\!_{4\to 4}}^{k}:=\frac{2}{2-C_N^{k}}\left(\left[\eta_{\psi}^{k}\right]^{2}+\left[\eta_c^{k}\right]^{2}+\tau\left[\eta_{D}^{k}\right]^{2}+\tau\left[\eta_{K}^{k}\right]^{2}\right)^{\frac{1}{2}},
		\end{equation}
		with
		\begin{equation*}
			\begin{aligned}
				\eta_{D}^{k}:=&\left\|\left(D\right)^{-\frac{1}{2}}\left(D\del\delta c_{h}^{k}+\bmu_w\delta c_{h}^{k}\right)\right\|,\\
				\eta_K^{k}:=&\left\|\left(K^{k}\right)^{-\frac{1}{2}}\left(\left(\delta K^{k}\right)\del(\psi_{h}^{k}+z)-(K\circ\theta)'(\psi_{h}^{k-1})\del(\psi_{h}^{k-1}+z)\delta\psi_{h}^{k}\right)\right\|,\\
				\eta_\psi^{k}:=&\left\|\left(\frac{\partial\theta^{k}}{\partial\psi}\right)^{-\frac{1}{2}}\left(\frac{\partial\theta^{k-1}}{\partial \psi}\delta\psi_{h}^{k}-\delta\theta^{k}\right)\right\|,\\
				\eta_c^{k}:=&\left\|\left(\frac{\partial\theta^{k}}{\partial c}\right)^{-\frac{1}{2}}\left(\frac{\partial\theta^{k-1}}{\partial c}\delta c_{h}^{k}-(\delta\theta^{k})c_{h}^{k}\right)\right\|.
			\end{aligned}
		\end{equation*}
	\end{lemma}
	\begin{proof}
		The proof follows similar steps as in the previous proof.
	\end{proof}
	\begin{remark}[Similarity with Richards' equation]
		We note that the estimators in \Cref{lem: surfactant transport Lscheme to newton} and \Cref{lem: surfactant transport newton to newton} correspond to the estimators in \cite{stokke_adaptive_2023} if $c=0$ and we disregard the equilibriated fluxes.
	\end{remark}
	
	\begin{remark}
		The alternating linearization procedure in \cite{illiano_iterative_2021}, could in principle use the estimators in \cite{stokke_adaptive_2023} with modifications to reflect the coupling of the transport. The goal would then be to switch one step in the linearization procedure at a time, e.g. switching Richards' to Newton while keeping the L-scheme for the transport if the guess is not sufficiently good enough for the concentration. This has not been pursued here, as we seek a global higher-order method. Numerical experiments in \cite{illiano_iterative_2021,zeng_multidimensional_2021} also indicate that the fully coupled solution strategy converges faster. But in principle, the decoupling strategy would consider one equation at a time. Here we could have, similarly to the two-phase flow and the Biot case, pursued criteria for adaptively choosing $L_1$ and $L_2$.
	\end{remark}
	
	\subsection{Adaptive algorithms}
	By considering the estimates in \Cref{sec: surfactant estimators} with the estimators  $\eta_{\!_{3\to 4}}$ and $\eta_{\!_{4\to 4}}$  defined in \eqref{eq: surf est 3to4}, and \eqref{eq: surf est 4to4} and the computational considerations in \Cref{sec: comput considerations}, we propose two adaptive iterative algorithms, a switching algorithm between the L-scheme and Newton's method, and an adaptive time-stepping algorithm for Newton's method:
	\begin{algorithm}[H]
		\caption{L-scheme/Newton \it{a-posteriori} switching}\label{alg: a-posteriori}
		\begin{algorithmic}
			\Require $\boldsymbol{\psi}^{n,0}, c^{n,0}$ as initial guess. 
			\Ensure Scheme=\fbox{L-scheme}, $C_{\rm tol}=1.5$
			\For{k=1,2,..}
			\If{Scheme=\fbox{L-scheme}} 
			\State
			Compute iterate using  \fbox{L-scheme}, i.e. \eqref{eq: surfactant transport Lscheme}        
			\If{$\eta_{\!_{3\to 4}}^k\leq C_{\rm tol} \eta^k_{\!_{inc,3}}$}
			\State Set Scheme=\fbox{Newton}
			\EndIf
			\Else 
			\State \  Compute iterate using \fbox{Newton}, i.e. \eqref{eq: surfactant transport Newtons method}
			\If{$\eta_{\!_{4\to 4}}^k> \eta^k_{\!_{inc,4}}$}
			\State Set Scheme=\fbox{L-scheme}
			\EndIf
			\EndIf
			\EndFor

		\end{algorithmic}
	\end{algorithm}
	\begin{algorithm}[H]
		\caption{Newton \it{a-posteriori} control of time step size}\label{alg: a-posteriori 4}
		\begin{algorithmic}
			\Require $\boldsymbol{\psi}^{n,0}, c^{n,0}$ as initial guess. 
			\Ensure Scheme=\fbox{Newton}
			\If{Convergence at previous time step reached after n iterations}
			\State Set time step size $\t=2\t$
			\EndIf
			\For{k=1,2,..}
			\State
			Compute iterate using  \fbox{Newton}, i.e. \eqref{eq: surfactant transport Newtons method}        
			\If{$\eta_{\!_{4\to 4}}^k\leq  1$}
			\State Continue with \fbox{Newton}
			\EndIf
			\If{$\eta_{\!_{4\to 4}}^k> 1$}
			\State Set time step size $\t=\frac{\t}{2}$
			\EndIf
			\EndFor

		\end{algorithmic}
	\end{algorithm}
	\subsection{Numerical examples}
	Here we perform a numerical example to test the algorithm for switching between Newton and the L-scheme (\Cref{alg: a-posteriori}), which will be referred to as the $L/N$ scheme, and Newton's method with adaptive time stepping (\Cref{alg: a-posteriori 4}), which will be referred to as the $N-\t(n)$ scheme. $n$ is the convergence requirement at the previous time step for increasing the time step size. We consider an example where it is expected that Newton's method will struggle, but also where $L_2$ has to be large to satisfy the assumptions to guarantee the convergence of the L-scheme. The example parameters can be found in \Cref{tab: example surfactant transport}. We consider a strictly unsaturated porous medium.
	The domain is given by $\Omega= \Omega_1\cup\Omega_2$ where $\Omega_1=[0,1]\times [1/4,1]$ and $\Omega_2=[0,1]\times [0,1/4)$. We consider an initial pressure head profile of 
	\begin{equation}
		\psi^{0}= \begin{cases}
			-2,\quad &\mbox{in }\Omega_1,\\
			-y-1/4,\quad &\mbox{in } \Omega_2,
		\end{cases}
	\end{equation}
	and an initial concentration of $c^{0}=1$. The following source terms are applied
	\begin{equation*}
		f_3=\begin{cases}
			0.06\cos((4/3)\pi y)\sin(x) &\mbox{in }\Omega_1,\\
			0\quad &\mbox{in } \Omega_2,
		\end{cases}
	\end{equation*}
	and $f_4=0$ everywhere. At the top boundary, we use a Dirichlet condition for the pressure head $\psi=-2$. For the concentration we impose $c=1$ at $[0,1]\backslash[0.25,0.75]\times 1$ and $c=4$ at $[0.25,0.75]\times 1$. On the rest of the boundary, no-flow conditions are used for both variables. 
	For the parametrization of $\theta$ and $K$ we use the modfied van Genuchten-Mualem model proposed in \cite{knabner_influence_2003}
	\begin{subequations}
		\begin{align}
			\begin{split}
				\theta(\psi,c)=&\,\begin{cases}
					\theta_r+(\theta_s-\theta_r)\left(\frac{1}{1+\left(-\alpha\gamma(c)\psi\right)^{n}}\right)^{\frac{n-1}{n}},\quad &\psi\leq 0,\\
					\theta_s,\quad &\psi >0,
				\end{cases}
			\end{split}\\
			\begin{split}
				K(\theta(\psi,c))=&\,\begin{cases}
					K_s\Theta_{{\rm e}}(\psi,c)^{\frac{1}{2}}\left(1-\left(1-\Theta_{{\rm e}}(\psi,c)^{\frac{n}{n-1}})\right)^{\frac{n-1}{n}}\right)^{2},\quad &\psi\leq 0,\\
					K_s,\quad &\psi >0,
				\end{cases}
			\end{split}
		\end{align}
	\end{subequations}
	where $\gamma(c)$ is the surface tension and $\Theta_{{\rm e}}(\psi,c)$ is the effective water content defined as
	\begin{align*}
		\gamma(c)&:= \frac{1}{1-b\log(c/a+1)},\\
		\Theta_{{\rm e}}(\psi,c) &:= \frac{\theta(\psi,c)-\theta_r}{\theta_s-\theta_r}.
	\end{align*}
	Also, $\theta_r$ and $\theta_s$ are the residual and saturated water content, respectively. The iterative process is stopped when
	\begin{equation}
		\|\psi_h^{k}-\psi_h^{k-1}\|\leq 10^{-6},\mbox{ and }\|c_h^{k}-c_h^{k-1}\|\leq 10^{-6},
	\end{equation}
	is reached.
	
	\begin{table}[h!]
		\centering
		\begin{tabular}{l l l }
			\toprule

			Parameter & & Value  \\
			\midrule
			$D$ & Diffusion constant & $1e-3$ \\
			$a$& Compound specific constant& 0.44\\
			$b$ &Compound specific constant &  0.0046 \\
			$\theta_r$& Residual water content& 0.026 \\
			$\theta_s$ &Saturated water content& 0.42\\
			$K_s$& Conductivity of fully saturated medium& 0.12\\
			$n$ &Soil specific constant &  $2.9$ \\
			$\alpha$& Soil specific constant& 0.551\\
			$L_1$ &Stabilization of $(\psi)$& 0.1\\
			$L_2$ &Stabilization of $(c)$ & 128\\
			\bottomrule
		\end{tabular}
		\caption{Test case: surfactant transport - Parameters}
		\label{tab: example surfactant transport}
	\end{table}
	
	\subsubsection{Comparison of convergence properties}
	The total number of iterations for the proposed algorithms $L/N$ and $N/\t(n)$ for $n=\{5,10\}$ for different mesh sizes is presented in \Cref{fig: surfactantconvergence}. For all mesh sizes apart from the coarsest, the $L/N$ algorithm uses the fewest number of iterations. The adaptive time-stepping algorithm $N/\t(10)$ uses the second fewest iterations, indicating that being more aggressive with when to increase the time step size is beneficial to reducing the total number of iterations. The number of failed iterations for both $N/\t(10)$ and $N/\t(5)$ reflects the convergence properties of Newton's method, that a smaller time step size is needed for finer meshes. In \Cref{fig: surfactant transport adaptive time step}, the evolution of the number of iterations and time step size for $h=2/\sqrt{60}$ is displayed. Newton's method fails initially 8 times before it converges. The difference in the number of time steps between the two adaptive time-stepping strategies is also significant, as the increase in time step size happens slower for $N/\t(5)$.
	\begin{figure}[H]
		\centering
		\includegraphics[width=0.55\textwidth]{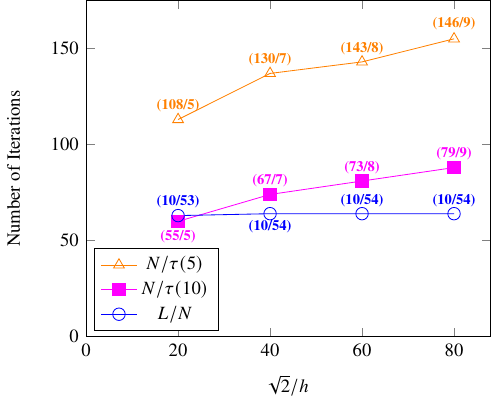}
		\caption{Test case: surfactant transport - Total number of iterations at $T=1$ for $\tau=0.1$ and varying mesh size. For $N/\t$ this means that the initial time step size is $\tau=0.1$. The number in blue parentheses corresponds to (number of L-scheme iterations/number of Newton iterations). The pink and orange parentheses correspond to (successful Newton iterations/unsuccessful Newton iterations). The L-scheme converges for all mesh sizes but uses more than 1000 iterations, and Newton's method converges when $\t=0.0025$ on the coarsest mesh, with a total number of iterations larger than 1000 iterations; therefore, they are omitted from the plot. Newton diverges on finer meshes unless an even smaller time step is chosen. }
		\label{fig: surfactantconvergence}
	\end{figure}
	\begin{figure}[H]
		\centering
		\includegraphics[width=0.55\textwidth]{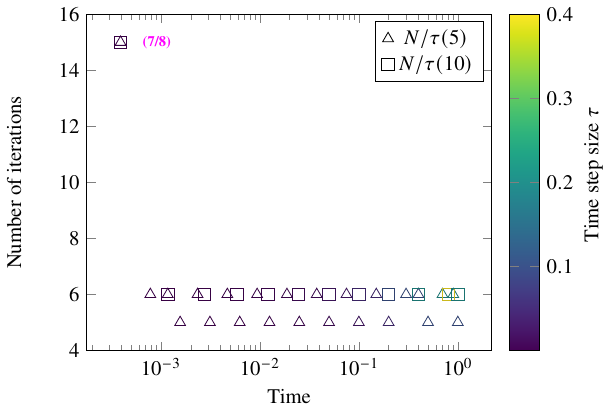}
		\caption{Test case: surfactant transport -  Number of iterations at each time step for $h=2/\sqrt{60}$ with initial time step size $\t=0.1$ for the adaptive time-stepping algorithm. The parentheses correspond to (successful Newton iterations/unsuccessful Newton iterations).}
		\label{fig: surfactant transport adaptive time step}
	\end{figure}
	\subsubsection{Switching characteristics}
	Based on \Cref{fig: surfactantconvergence}, it is evident that the $L/N$-scheme only required 1 L-scheme iteration before switching to Newton's method. Therefore, we choose to take a closer look at the adaptive time-stepping algorithm and the estimator $\eta_{\!_{4\to4}}^{k}$ which predicts the success and failure of Newton's method. The evolution of the estimator and efficiency index for the first two time steps for $N/\t(10)$ are displayed in \Cref{fig: surfactantSwicthing}. First note that for the first iteration at the second time step $t_2$, the ratio between the estimator and the current incremental error is larger than one. In the two-phase flow algorithm, this would have implied divergence. Here, since we made the choice to only increase when $\eta_{\!_{4\to 4}}>1$, it continues with the same time step size. Since the estimate is only an upper bound on the incremental error at the next iteration, and does not imply that it will not be larger than the current incremental error we should be careful when using a direct comparison between $\eta_{\!_{4\to 4}}$ and $\eta_{\!_{inc,4}}$. The $C_{\rm tol}$, which originally was motivated to expedite switching between different schemes, can also be used in this context to avoid a strict comparison. Further, the efficiency indices do not stay larger than 1 for all iterations. This is due to the choice of not computing the constant $C_N^{k}$, which means that we do not have a guaranteed upper bound. Despite this, the estimator for Newton works well in practice to guide the adaptive time-stepping. 
	
	\begin{figure}[H]
		\centering
		\includegraphics[width=0.5\textwidth]{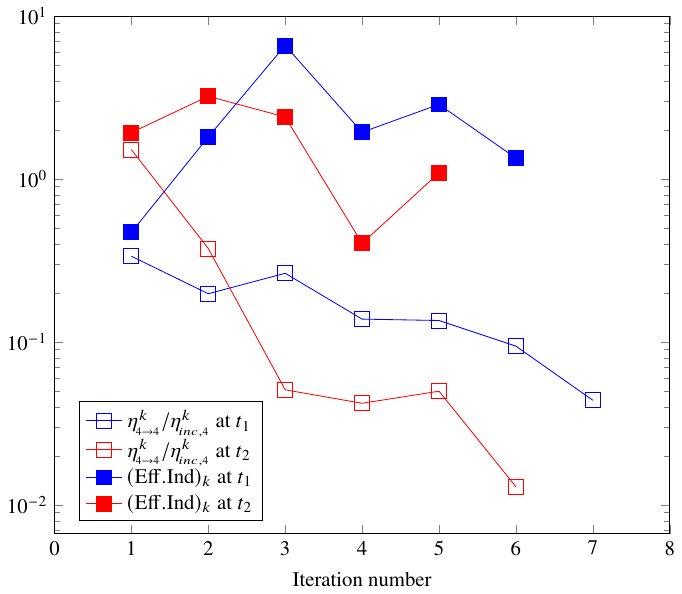}
		\caption{Test case: surfactant transport - Evolution of switching indicators for $N/\t(10)$, i.e. $\eta_{\!_{4\to 4}}^{k}/\eta_{\!_{inc,4}}^{k}$, for the first two time steps are plotted. The corresponding efficiency indices \eqref{eq: eff ind} of the estimator are also plotted.}
		\label{fig: surfactantSwicthing}
	\end{figure}
	\section{Quasi-static Biot model}\label{sec: biot}
	
	In this section, we consider the quasi-static Biot system, which describes a flow in a fully saturated deformable porous medium. The system can be stated as; Find the fluid pressure ($p$) and the displacement ($\bmu$) such that
	\begin{subequations}\label{eq: biot}
		\begin{align}
			c_0\p_t p+\alpha \div\p_t\bmu -\div (\k (\del p-\rho\bm{g})) =&\, f_5,\quad&&\mbox{in }\Omega\times[0,T],\\
			-\div\left(2\mu \e(\bmu)+\lambda \div\bmu\bm{I}\right)+\alpha\del p =&\, \bm{f}_6,\quad&&\mbox{in }\Omega\times[0,T],
		\end{align}
	\end{subequations}
	where $c_0$ is the specific storativity constant, $\alpha$ is the Biot coefficient, $\k$ is the permeability, $\rho$ is the fluid density, $\bm{g}$ is the gravitational vector $\mu$ and $\lambda$ are the Lame parameters. Further, $\e(\bmu):=\left(\del\bmu+\del\bmu^{T}\right)/2$ is the linearized strain, $f_5$ and $\bmF_6$ are source terms representing fluid injection or extraction and body forces, respectively.
	\begin{assumption}\label{as:auxiliary_functionsbiot}
		For the material properties $\k,\alpha,\rho,\mu,\lambda$, and $c_0$, and source terms $f_5(\cdot),\bm{f}_6(\cdot)$ in \eqref{eq: biot}, the following assumptions are made:
		\begin{itemize}
			\item[(a)] The parameters $\alpha,\mu,\k$ are strictly positive and $\l,c_0$ are non-negative. Let $\rho\in\R,\bm{g}\in\R^{d}$ be constant in time. 
			\item[(b)] The source functions satisfy $f_5\in C(0,T;L^2(\Om))$ and $\bm{f}_6\in C(0,T;\bm{L}^2(\Om))$.
		\end{itemize}
	\end{assumption}
	
	The assumption above on the material parameters is satisfied for realistic applications. For the discretization in time of \eqref{eq: biot}, we will use the backward Euler method, and in space the inf-sup stable Taylor-Hood elements. The fully discrete problem at time $t_n$ is then: Given $(p^{n-1}_{h},\bmu_{h}^{n-1})\in (Q_h,\bm{V}_h)$, find $(p^{n}_{h},\bmu_{h}^{n})\in (Q_h,\bm{V}_h)$ such that
	\begin{subequations}\label{eq: biot semi-discrete}
		\begin{align}
			\langle c_0(p^{n}_{h}-p^{n-1}_{h}),q_h\rangle+\langle \alpha \div (\bm{u}_{h}^{n}-\bm{u}^{n-1}_{h}),q_h\rangle +\tau \langle \k(\del p^{n}_{h}-\rho\bm{g}),\del q_h\rangle &= \tau\langle f_5^{n},q_h\rangle, \quad &&\forall q_h\in Q_h,\label{eq: biot discrete a}\\
			\langle 2\mu\e(\bmu^{n}_h),\e(\bmv_h)\rangle + \langle \lambda\div \bmu_{h}^{n},\div \bmv_h\rangle+\langle \alpha \del p_h^{n},\bmv_h\rangle &= \langle \bmF_6^{n},\bmv_h\rangle,\quad &&\forall \bmv_h\in\bm{V}_h. \label{eq: biot discrete b}
		\end{align}
	\end{subequations}
	To solve the system, we will use the fixed-stress splitting scheme. It is based on freezing the mean stress, or it can be viewed as adding a stabilization term to the flow equation, which allows for sequentially solving the flow and mechanics subproblems. The convergence of the fixed-stress has been studied intensively \cite{both_robust_2017,mikelic_convergence_2013,kim2011,castelletto_accuracy_2015,bause_spacetime_2017,both_numerical_2019}. A central focus has been on determining the stabilization parameter that uses the fewest number of iterations. A brute-force approach based on optimization of the parameter on coarse meshes, along with theoretically justified bounds, was proposed in \cite{storvik_optimization_2019}.
	
	Our goal in this section is to adaptively choose the parameter $L$, which arises in the stabilization term of the fixed-stress, based on \textit{a posteriori} error estimates. We know from \cite{storvik_optimization_2019} that the optimal $L$ should be within the bound
	\begin{align}\label{eq: biot bound on L}
		L\in\left[\frac{\alpha^{2}}{4\mu+2\lambda},\frac{\alpha^{2}}{K_{dr}}\right),
	\end{align}
	with $K_{dr}=\frac{2\mu}{d}+\lambda$ being the drained bulk modulus where $d$ is the dimension. It is common to view the tuning of $L$ as tuning $K_{dr}$, but for simplicity, we tune $L$. In contrast to the two-phase example, where we wanted the smallest $L$ possible, we know that the optimal parameter may be close to one of the endpoints or somewhere in the middle depending on the problem. This makes designing a good adaptive algorithm for tuning the parameter more difficult.  
	
	
	We consider two residuals since the system is decoupled.
	\begin{definition}[Residual quasi-static Biot]
		For all $(q_h,\bmv_h)\in (Q_h,\bm{V}_h)$ let the flow and mechanics residuals be defined as
		\begin{subequations}
			\begin{align}
				\begin{split}
					\langle\mathcal{R}_{\rm flow}(\varphi_1,\bm{\varphi}_2),(q_h,\bmv_h)\rangle:=&\, 	\langle c_0 (\varphi_1-p^{n-1}_{h}),q_h\rangle +\tau\langle k(\del \varphi_1-\rho\bm{g}),\del q_h\rangle\\&+\langle \alpha \div (\bm{\varphi}_2-\bm{u}^{n-1}),q_h\rangle-\tau\langle f_5,q_h\rangle,
				\end{split}\\
				\langle\mathcal{R}_{\rm mech}(\varphi_1,\bm{\varphi}_2),(q_h,\bmv_h)\rangle:=&\, 	\langle 2\mu\e(\bm{\varphi}_2),\e(\bmv_h)\rangle + \langle \lambda\div \bm{\varphi}_2,\div \bmv_h\rangle +\langle \alpha \del \varphi_1,\bmv_h\rangle - \langle \bmF_6,\bmv_h\rangle.
			\end{align}
		\end{subequations}
		\label{def: biot residual} 
	\end{definition}
	
	\subsection{Fixed-stress splitting}
	The fixed-stress algorithm can be expressed as; For a stabilization parameter $L>0$ given $p^{n-1}_{h},p^{k}_{h}\in Q_h$, and $\bmu_{h}^{k}\in\bm{V}_h$ find $p^{k+1}_{h}\in Q_h$ such that
	\begin{subequations}\label{eq: biot fixed-stress full formulation}
		\begin{align}
			\begin{split}
				\langle c_0 (p^{k+1}_{h}-p^{n-1}_{h}),q_h\rangle +\langle L\delta p^{k+1}_{h},q_h\rangle+\tau\langle (k\del p^{k+1}_{h}-\rho\bm{g}),\del q_h\rangle&\\+\langle \alpha \div (\bmu_{h}^{k}-\bm{u}^{n-1}),q_h\rangle&=\langle f_5^{n},q_h\rangle,
			\end{split}
		\end{align}
		for all $q_h\in Q_h$. Next solve the mechanincs equation: Given $(p^{k+1}_{h},\bmu_{h}^{k})\in (Q_h,\bm{V}_h)$ find $\bmu_{h}^{k+1}\in \bm{V}_h$ such that
		\begin{align}
			\langle 2\mu\e(\bmu_{h}^{k+1}),\e(\bmv_h)\rangle+\langle \lambda \div \bmu_{h}^{k+1},\div\bmv_h\rangle+\langle \alpha\del p^{k+1}_{h},\bmv_h\rangle=\langle \bm{f}_6^{n},\bmv_h\rangle,
		\end{align}
		for all $\bmv_h\in\bm{V}_h$.
	\end{subequations}
	Based on \eqref{eq: biot fixed-stress full formulation}, we can define two separate bilinear forms by
	\begin{subequations}\label{eq: biot bilinear form}
		\begin{align}
			\mathcal{B}_{\rm flow}(\varphi,q_h):=&\, \langle (c_0+L)\varphi_1,q_h\rangle +\tau \langle \del\varphi_1,\del q_h\rangle,\\
			\mathcal{B}_{\rm mech}(\bm{\varphi}_2,\bmv_h):=&\, \langle 2\mu\e(\bm{\varphi}_2),\e(\bmv_h)\rangle + \langle \lambda\div \bm{\varphi}_2,\div \bmv_h\rangle.
		\end{align}
	\end{subequations}
	The fixed-stress algorithm \eqref{eq: biot fixed-stress full formulation} can then be expressed in the following way. 
	\begin{iterative}[Fixed-stress] For the bilinear forms defined in \eqref{eq: biot bilinear form} and the residuals in \Cref{def: biot residual} the fixed-stress splitting can be defined as
		\begin{subequations}\label{eq: biot fixed stress}
			\begin{align}
				\mathcal{B}_{\rm flow}(\delta p^{k+1}_{h},q_h)&=-\langle\mathcal{R}_{\rm flow}(p^{k}_{h},\bmu_{h}^{k}),q_h\rangle,
				\\
				\mathcal{B}_{\rm mech}(\delta\bmu_{h}^{k+1},\bmv_h)&=-\langle\mathcal{R}_{\rm mech}(p^{k+1}_{h},\bmu_{h}^{k}),\bmv_h\rangle.
			\end{align}
		\end{subequations}
	\end{iterative}
	Since the equations are decoupled, we consider the iteration-dependent norms separately, instead of considering the entire system at once as we have done previously.
	\begin{definition}[Iteration dependent norms for fixed-stress] For $\varphi_1\in H_{0}^{1}(\Omega)$ and $\bm{,\varphi}_2\in\bm{H}^{1}_{0}(\Omega)$, the iteration-dependent norms for the fixed-stress splitting \eqref{eq: biot fixed stress} are defined by
		\begin{subequations}
			\begin{align}
				\norm{\varphi_1}_{\rm flow}:=&\,\left(\int_{\Omega} (c_0+L)(\varphi_1)^{2}+\tau |k^{\frac{1}{2}}\del\varphi_1|^{2}\right)^{\frac{1}{2}},\\
				\norm{\bm{\varphi}_2}_{\rm mech}:=&\,\left(\int_{\Omega} 2\mu\e(\bm{\varphi}_2):\e(\bm{\varphi}_2)+\lambda(\div\bm{\varphi}_2)^{2}\right)^{\frac{1}{2}}.
			\end{align}
		\end{subequations}
	\end{definition}
	\begin{remark}[Incremental error for decoupled systems]
		In general, it is possible to use only one equation to assess the performance of the iterative method. This would mean only deriving an estimate for the incremental error of the flow equation, but since it does not reflect the convergence rate of the mechanics equation, we chose to define the incremental error as 
		\begin{equation}
			\eta_{\!_{{\rm inc},5}}^{k}:=\norm{\delta p^{k}_{h}}_{\rm flow}+\norm{\delta \bmu_{h}^{k}}_{\rm mech}.
		\end{equation}
	\end{remark}
	
	\subsection{Estimators}
	Here, we derive the estimators for the adaptive algorithm.
	\begin{lemma}\label{lem: biot}
		Let $\{p^{k}_{h},\bmu_{h}^{k}\}$ be a sequence of iterates generated using the fixed-stress \eqref{eq: biot fixed stress}, then the error of the iterative decoupling satisfies
		\begin{equation}\label{eq: biot final estimate}
			\norm{\delta p^{k+1}_{h}}_{\rm flow}+\norm{\delta \bmu_{h}^{k+1}}_{\rm mech}\leq \eta_{\!_{5\to 5}}^{k},
		\end{equation}
		where $\eta_{\!_{5\to 5}}^{k}=\eta_{\rm flow}^{k}+\eta_{\rm mech}^{k}$ and
		\begin{subequations}
			\begin{align}
				\eta_{\rm flow}^{k}:=&\,(\|(L+c_0)^{-\frac{1}{2}}L\delta p^{k}_{h}\|^{2}+\|(\t\k)^{-\frac{1}{2}}\alpha \delta\bmu_{h}^{k}\|^{2})^{\frac{1}{2}},\\
				\eta_{\rm mech}^{k}:=&\,\left(\left\| \lambda^{-\frac{1}{2}}\alpha\left(\frac{\alpha c_0^{-1}}{\alpha c_0^{-1}+2K_{dr}}\right)^{\frac{1}{2}}\delta p^{k}_{h}\right\|^{2}\right)^{\frac{1}{2}}.
			\end{align}
		\end{subequations}
	\end{lemma}
	\begin{proof} First, we consider the flow part, and perform steps 1-3. Note that the bilinear form is symmetric, and therefore the fourth step is not needed.
		\begin{equation}
			\begin{aligned}
				\norm{\delta p^{k+1}_{h}}_{\rm flow}^{2}=&
				\int_{\Omega} (c_0+L)(\delta p^{k+1}_{h})^{2}+\tau |k^{\frac{1}{2}}\del \delta p^{k+1}_{h}|^{2}\\
				=&-\langle\mathcal{R}_{\rm flow}(p^{k}_{h},\bmu_{h}^{k}),\delta p^{k+1}_{h}\rangle\\
				=&-\langle\mathcal{R}_{\rm flow}(p^{k}_{h},\bmu_{h}^{k}),\delta p^{k+1}_{h}\rangle+\langle\mathcal{R}_{\rm flow}(\bmu_{h}^{k-1},p^{k-1}_{h}),\delta p^{k+1}_{h}\rangle+\mathcal{B}_{\rm flow}(\delta p^{k}_{h},\delta p^{k+1}_{h})\\
				=& \langle L \delta p^{k}_{h},\delta p^{k+1}_{h}\rangle+\langle \alpha \delta\bmu_{h}^{k},\del\delta p^{k+1}_{h}\rangle
				\\
				\leq& \|(L+c_0)^{-\frac{1}{2}}L\delta p^{k}_{h}\|
				\|(L+c_0)^{\frac{1}{2}}\delta p^{k+1}_{h}\|+\|(\t\k)^{-\frac{1}{2}}\alpha \delta\bmu_{h}^{k}\|
				\|\t^{\frac{1}{2}}\k^{\frac{1}{2}}\del\delta p^{k+1}_{h}\|\\
				\leq&\eta_{\rm flow}^{k}\norm{\delta p^{k+1}_{h}}_{\rm flow}.
			\end{aligned}
		\end{equation}
		Next, we bound the iteration-dependent mechanics norm. It follows similar steps, but in contrast to the previous proofs, we use an existing result from the literature for the fixed-stress, see e.g. \cite{mikelic_convergence_2013}, that
		\begin{align}\label{eq: biot convergence}
			\|\delta p^{k+1}_{h}\|^{2}\leq \frac{\alpha c_0^{-1}}{\alpha c_0^{-1}+2K_{dr}}\|\delta p^{k}_{h}\|^{2},
		\end{align}
		to deal with the coupling term. Following the same procedure as above, we get
		\begin{align*}
			\norm{\delta \bmu_{h}^{k+1}}_{\rm mech}^{2}=&\,\int_{\Omega} 2\mu\e(\delta \bmu_{h}^{k+1}):\e(\delta \bmu_{h}^{k+1})+\lambda(\div\delta \bmu_{h}^{k+1})^{2}\\
			\leq&\,\| \lambda^{-\frac{1}{2}}\alpha\delta p^{k+1}_{h}\|\|\lambda^{\frac{1}{2}}\div\delta\bmu_{h}^{k+1}\|\\
			\overset{\eqref{eq: biot convergence}}{\leq}&\,\left\| \lambda^{-\frac{1}{2}}\alpha\left(\frac{\alpha c_0^{-1}}{\alpha c_0^{-1}+2K_{dr}}\right)^{\frac{1}{2}}\delta p^{k}_{h}\right\|\|\lambda^{\frac{1}{2}}\div\delta\bmu_{h}^{k+1}\|\\
			\leq& \eta_{\rm mech}^{k}\norm{\delta\bmu_{h}^{k+1}}_{\rm mech}.
		\end{align*}
		Combining the two estimates yields the final estimate \eqref{eq: biot final estimate}.
	\end{proof}
	\begin{remark}
		[Extension to different stabilizations, including nonlinear poro-mechanical models]
		In the case of fixed-stress applied to non-linear permeability of the type $K(\div\bmu)$ which is proven to converge in \cite{kraus_fixed-stress_2024}, the estimator in \Cref{lem: biot} only needs to be modified to have a term of the type $(k(\div \bmu_{h}^{k})-k(\div \bmu_{h}^{k-1}))\del p^{k}_{h}$ in $\eta_{\rm flow}$. In the non-linear extension \cite{borregales_robust_2018} where there are two non-linear functions $b(p)$ and $\lambda(\div\bmu)$ with two stabilization parameters, each equation can be viewed separately to tune both $L_1$ and $L_2$. The estimator $\eta_{\rm flow}$ for $L_1$ would be similar to \Cref{lem: biot} with a dependence on $b(p^{k}_{h})$. In the stabilization proposed in \cite{pe_de_la_riva_oscillation-free_2025}, being an iterative decoupling method closely related to the fixed-stress, with a different tuneable parameter $\gamma$, can also be written in the form \eqref{eq: general linearization}. 
	\end{remark}

	\subsection{Adaptive algorithm}
	
	For the following, we denote $L_{start}$ as the starting value of the adaptive scheme. Based on the bound \eqref{eq: biot bound on L} we will choose the smallest $L$ possible initially, such that we expect to increase the stabilization parameter.  For simplicity we will either decrease or increase $L$ based on what happened in the previous time step. Another point is how much we should increase or decrease the parameter. Since we start with a small $L$, we want to allow for a more aggressive increase, but a more conservative decrease. When evaluating the estimator $\eta_{\!_{5\to 5}}^{k}$ we saw that the efficiency of the estimator was high, as will be seen in the subsequent subsection. This makes creating a criterion for when to tune $L$ more difficult, but first note that a high efficiency index can be interpreted as the convergence is faster than expected. Secondly, another consequence of the high efficiency of the estimator is that the ratio $\eta_{\!_{5\to 5}}^k/\eta_{\!_{inc,5}}^{k}$ can be larger than one despite the fact that it converges. Based on these observations, we propose the following algorithm.
	\begin{algorithm}[H]
		\caption{Adaptive Fixed-Stress}\label{alg: a-posteriori 5}
		\begin{algorithmic}
			\Require $p^{n,0},\bmu^{n,0}$ as initial guess. $L=L_{\rm start}$, choose $C_{\rm inc}$, bool increase = true, bool hasIncreased = false.
			\Ensure Scheme=\fbox{Fixed-Stress}
			\If{hasIncreased==true}
			\State Increase=false
			\Else
			\State Increase =true
			\EndIf
			\For{k=1,2,..}
			\State
			Compute iterate using  \fbox{Fixed-Stress}, i.e.  \eqref{eq: biot fixed stress}       
			\If{$\eta_{\!_{5\to 5}}^k\geq  10\eta_{\!_{inc,5}}^{k}$ and $({\rm Eff. Ind})_k<100$}
			\If{Increase==true}
			\State $L=\min\{L_{\rm phys},C_{\rm inc}L\}$, hasIncreased=true
			\Else \State $L=\max\{L_{\rm min},0.9L\}$,\, hasIncreased=false
			\EndIf
			\EndIf
			\EndFor
		\end{algorithmic}
	\end{algorithm}
	\subsection{Numerical examples}
	
	The following example is inspired by \cite{bause_spacetime_2017, both_numerical_2019}. We consider an L-shaped domain $\Omega= (0,1)^{2}\backslash [0.5,1]^{2}$ with $h=\sqrt{2}/80$ on the time interval $(0,0.5)$ with a fixed time step size $\t=0.01$. The initial pressure and the initial displacement are zero, i.e. $p(\cdot,0)=0$ and $\bmu(\cdot,0)=0$. On top of the domain, we consider a time-dependent traction force $\bm{t}=(0,-256h_{\rm max}t^{2}(t-0.5)^{2})^{T}$ with $h_{max}=1e10$. On the top, we also impose a homogeneous Dirichlet condition for the pressure $p=0$. On the remaining flow boundary, we use a no-flow condition. For the displacement at the boundaries, we enforce zero normal displacement on the left, bottom, and in the cut of the L-shaped domain, and finally, on the lower right side, a zero traction condition. We consider a medium and fluid with the following properties: Young's modulus $E=1e11$, storage coefficient $c_0=1e-11$, permeability $k=1e-13$, and the Biot coefficient $\alpha=0.9$. Gravitational effects are ignored, i.e. $\bm{g}=0$. We study the performance for the Poisson's ratio $\nu\in\{0.01,0.2,0.4\}$. The adaptive fixed-stress splitting (\Cref{alg: a-posteriori 5}) will in this section be denoted $L_A(C_{\rm inc})$ with $C_{\rm inc}$ being the constant with which the $L$ is increased. We consider a wide variety of stabilization parameters $L$ commonly used in the literature, including the endpoints of the bound on $L$, see \Cref{tab: biot parameters} for their names and values. In particular, we mention $L_{1D}$, which is known to be good for this problem, especially for a higher Poisson's ratio.
	\renewcommand{\arraystretch}{1.3} 
	\begin{table}[H]
		\centering
		\begin{tabular}{c|ccccccc}
			Name&$L_{\rm min}$  &  $L_{\rm phys}$&$L_{\rm MW}$&  $L_{\rm 1D}$ & $L_{opt}(0.01)$& $L_{opt}(0.2)$& $L_{opt}(0.4)$\\
			\hline
			Value & $\frac{\alpha^2}{4\mu+2\lambda}$ & $\frac{\alpha^2}{K_{dr}}$ & $\frac{\alpha^2}{2K_{dr}}$ & $\frac{\alpha^2}{2\mu+\lambda}$& $2.5L_{\rm min}$ &$2.3L_{\rm min}$&$L_{\rm 1D}$\\
		\end{tabular}
		\caption{Name and value of specific stabilization parameters. The $L_{opt}$ is a brute force calculated optimal $L$ for the first time step on a coarse mesh for each $\nu$.}
		\label{tab: biot parameters}
	\end{table}
	\noindent We use the same relative error stopping criterion as \cite{storvik_optimization_2019}
	\begin{equation*}
		\|p^{k}_{h}-p^{k-1}_{h}\|\leq 10^{-6}\|p^{k}_{h}\|.
	\end{equation*}
	\subsubsection{Comparison of convergence properties}
	The total number of iterations for each $L$ is displayed in \Cref{tab: biot}. The choice $L_{\rm min}$, being the smallest value, results in the largest number of iterations across all $\nu$ considered. $L_{\rm MW}$ and $L_{\rm phys}$ uses a similar amount of iterations for $\nu=0.01$ and $\nu=0.2$, with $L_{\rm phys}$ using fewer for the largest Poisson ratio. As expected for this example $L_{1D}$ uses fewer iterations than the aforementioned parameters and also uses fewer iterations than $L_{A}(1.25)$ for all $\nu$. For smaller Poisson ratio the adaptive algorithms $L_{A}(1.3)$ and $L_{A}(1.4)$ uses fewer iterations than $L_{1D}$. However, for $\nu=0.4$ all the adaptive algorithms use the largest number of iterations. The brute-force optimized $L_{opt}$ uses the fewest number of iterations out of all choices of $L$.
	\renewcommand{\arraystretch}{1.07} 
	\begin{table}[H]
		\centering
		\begin{tabular}{c|ccc}
			\hline
			$\nu$ & 
			0.01 & 
			0.2 & 
			0.4
			\\
			\hline
			Scheme  &
			Tot. Itr & 
			Tot. Itr & 
			Tot. Itr 
			\\
			\hline
			$L_{\rm min}$ &  2413 & 812 & 438  \\ 
			$L_{\rm MW}$ &  576 & 488 & 399  \\
			$L_{\rm phys}$ & 593 & 476 & 320  \\
			$L_{\rm 1D}$ &  568 & 386 & 247 \\
			$L_A(1.25)$& 589& 398&438\\
			$L_A(1.3)$ &  501 & 353 & 438  \\  
			$L_A(1.4)$& 491& 349 &438\\
			$L_{opt}$ & 465& 341&247\\
			\hline
		\end{tabular}
		\captionsetup{width=\textwidth}
		\caption{Test case: Biot - Total number of iterations at $T=0.5$ with $\tau=0.01$ for different stabilization parameters for varying Poisson's ratio $\nu$.}
		\label{tab: biot}
	\end{table}

	\subsubsection{Switching characteristics}
	
	In \Cref{fig: biot switching} the evolution of the switching indicators along with the effectivity indices is plotted for the first time step, both for the constant $L_{opt}$ in (a) and also for the adaptive algorithm $L_A(1.4)$ (b). The optimal parameter has a very high initial efficiency index, indicating that the bound in \Cref{lem: biot} is not sharp. But a high efficiency index can be interpreted as the convergence is better than expected. In addition, for $L_{opt}$ in view of \Cref{alg: a-posteriori 5} no tuning would have happened. For the adaptive algorithm $L_A(1.4)$, the $L$ is increased three times for the two smallest values of $\nu$ before it stops increasing in the first time step. In the case when $\nu=0.4$, no changes in $L$ happens and it performs like $L_{min}$. From \Cref{tab: biot} it is evident that the same phenomena occur for all the adaptive algorithms.
	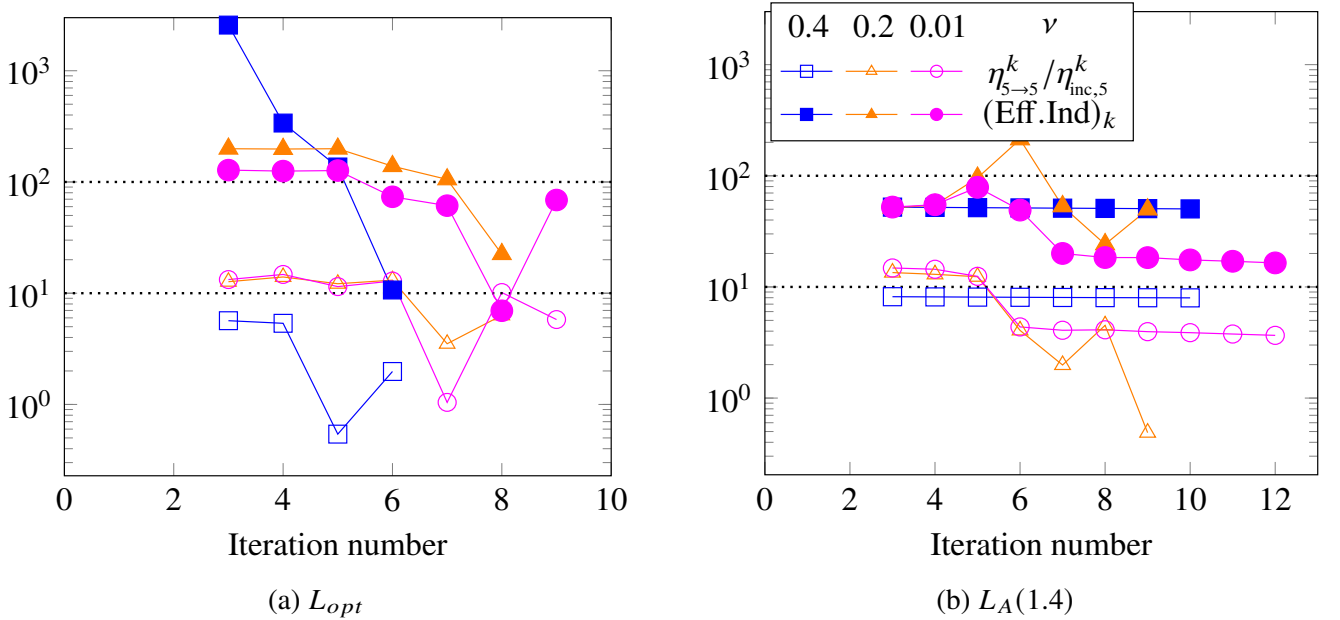
\begin{figure}[H]
		\centering
		
		\tikzstyle{every node}=[font=\small]
		\subfloat[][$L_{opt}$]{\resizebox{0.45\textwidth}{!}{
				\begin{tikzpicture} 
					\begin{semilogyaxis}
						[ width=8cm,
						height=7cm,
						xlabel={ Time},
						xmin = 0,
						xmax = 10,
						ymax = 3000,
						xlabel=Iteration number,legend style={at={(0.62,0.875)},anchor=west,legend columns=4,font=\small}]

						\addplot[color=blue,mark=square,mark options={scale=1.5}] 
						coordinates {  
							
							(3,5.660916909	)
							(4,5.363346092	)
							(5, 0.5409518992)
							(6,1.978216848)
						}; 
						
						\addplot[color=orange,mark=triangle,mark options={solid,scale=1.5}] 
						coordinates {  
							
							(3,12.68688552	)
							(4,13.99244142	)
							(5,12.16629476)
							(6,13.10498614	)
							(7,3.513460756	)
							(8,6.324533926)
						}; 
						
						\addplot[color=colorL1,mark=o,mark options={solid,scale=1.5}] 
						coordinates {  
							
							(3,13.26145839	)
							(4,14.7451385	)
							(5, 11.46241314)
							(6,12.86891631	)
							(7,1.043425397	)
							(8,10.13866067	)
							(9,5.796350233)
						}; 
						
						
						\addplot[color=blue,mark=square*,mark options={scale=1.5}] 
						coordinates {  
							(3,2572.560654	)
							(4,338.6233975	)
							(5,136.3532907)
							(6,10.68531033)

						}; 
						
						\addplot[color=orange,mark=triangle*,mark options={solid,scale=1.8}] 
						coordinates {  
							(3,199.01405	)
							(4,197.78646	)
							(5,199.0898858)
							(6,138.6683621	)
							( 7, 105.4540527	)
							(8, 22.43408708)

						}; 
						
						\addplot[color=colorL1,mark=*,mark options={solid,scale=1.8}] 
						coordinates {  
							(3,127.8125291	)
							(4,125.2554983	)
							(5,126.6591023)
							(6,73.60486581	)
							(7, 61.24862823	)
							(8,6.956185209	)
							(9, 68.83253943)
							
						};


						\addplot[thick,dotted] coordinates {(0,10) (10,10)};
						\addplot[thick,dotted] coordinates {(0,100) (10,100)};
						
					\end{semilogyaxis} 
				\end{tikzpicture}
				
		}}\hspace{0.02\textwidth}
		\subfloat[][$L_A(1.4)$]{\resizebox{0.435\textwidth}{!}{
				\begin{tikzpicture} 
					\begin{semilogyaxis}
						[ width=8cm,
						height=7cm,
						xlabel={ Time},
						xmin = 0,
						xmax = 13,
						ymax = 3000,
						xlabel=Iteration number,legend style={at={(0.01,0.87)},anchor=west,legend columns=4,font=\small}] 
						
						\addlegendentry{\hspace{-20pt}{0.4}}
						\addlegendimage{empty legend}
						\addlegendentry{\hspace{-20pt}{0.2}}
						\addlegendimage{empty legend}
						\addlegendentry{\hspace{-20pt}{0.01}}
						\addlegendimage{empty legend}
						\addlegendentry{$\nu$}
						\addlegendimage{empty legend}
						\addlegendimage{color=blue,mark=square}
						\addlegendentry{}
						\addlegendimage{color=orange,mark=triangle}
						\addlegendentry{}
						\addlegendimage{color=colorL1,mark=o}
						\addlegendentry{}
						\addlegendimage{empty legend}
						\addlegendentry{$\eta_{\!_{5\to 5}}^{k}/\eta_{\!_{{\rm inc},5}}^{k}$}
						\addlegendimage{color=blue,mark=square*}
						\addlegendentry{}
						\addlegendimage{color=orange,mark=triangle*}
						\addlegendentry{}
						\addlegendimage{color=colorL1,mark=*}
						\addlegendentry{}
						\addlegendimage{empty legend}
						\addlegendentry{$({\rm Eff. Ind})_k$}
						\addplot[color=blue,mark=square,mark options={scale=1.5}] 
						coordinates {  
							
							(3,8.166760162	)
							(4,8.133868637	)
							(5,8.099289411)
							(6,8.066747826	)
							(7,8.038266143	)
							(8,8.014680301	)
							(9,7.995037647	)
							(10,7.974897086)
						}; 
						
						\addplot[color=orange,mark=triangle,mark options={solid,scale=1.5}] 
						coordinates {  
							
							(3,13.46671486	)
							(4,12.98439806	)
							(5,12.35530352)
							(6,4.06352847	)
							(7,1.985037361	)
							(8,4.510753787	)
							(9,0.4884641553)
						}; 
						
						\addplot[color=colorL1,mark=o,mark options={solid,scale=1.5}] 
						coordinates {  
							
							(3,14.80960382	)
							(4,14.42048126	)
							(5, 12.45247489)
							(6,4.367078145	)
							(7,4.075044719	)
							(8,4.11479117	)
							(9,3.959292564	)
							(10,3.877752216)
							(11,3.769760625	)
							(12,3.666703376)
						};

						\addplot[color=blue,mark=square*,mark options={scale=1.5}] 
						coordinates {  
							(3,52.23688664	)
							(4,51.95730732	)
							(5,51.66925295)
							(6,51.38992167	)
							(7,   51.1264074	)
							(8,  50.87770413	)
							(9, 50.63013942	)
							(10,50.34614654)

						}; 
						
						\addplot[color=orange,mark=triangle*,mark options={solid,scale=1.8}] 
						coordinates {  
							(3,51.56817062	)
							(4,54.46336997	)
							(5,95.87855058)
							(6,212.6320893	)
							(7,  53.13605876	)
							(8,  24.15653211	)
							(9,  49.67889122)

						}; 
						
						\addplot[color=colorL1,mark=*,mark options={solid,scale=1.8}] 
						coordinates {  
							(3,52.28510978	)
							(4,54.91664277	)
							(5,78.75443773)
							(6,49.1286121	)
							(7,  20.00838375	)
							(8,18.39202199	)
							(9, 18.35721168	)
							(10,17.49724273)
							(11,17.01455333	)
							(12,16.45161232)
							
						};

						\addplot[thick,dotted] coordinates {(0,10) (13,10)};
						\addplot[thick,dotted] coordinates {(0,100) (13,100)};
						
					\end{semilogyaxis} 
		\end{tikzpicture}    }}
		
		\caption{Test case: Biot - Evolution of switching indicators for the constant $L_{opt}$ (a) and adaptive $L_A(1.4)$ (b) for different $\nu$ along with the efficiency indices \eqref{eq: eff ind} during the first time step for $\t=0.01$ and $h=\sqrt{2}/80$. The efficiency index is plotted at the iteration when it is available, since the previous estimator can not be evaluated before the next iteration is computed.}
		\label{fig: biot switching}
	\end{figure}
	
	\section{Conclusions}
	In this paper, we have provided a framework for the design of convergent adaptive iterative algorithms for solving multi-physics problems. This includes adaptive switching between two linearization schemes, adaptive tuning of a stabilization parameter, and adaptive time stepping. The algorithms are based on \textit{a posteriori} error estimators for the incremental error (linearization or iterative error) of a numerical scheme. The estimators are computationally inexpensive.  The methodology has been tested on three problems: two-phase flow in porous media, surfactant transport in porous media, and the quasi-static Biot system.  
	
	\section*{Acknowledgments}
	The authors acknowledge the support of the VISTA program, The
	Norwegian Academy of Science and Letters and Equinor. FAR wants to thank the support from the project MUPSI, CETP-2023-00298.
	
	\bibliographystyle{elsarticle-num} 
	\bibliography{ref}	
	
\end{document}